\documentclass[letterpaper]{article}

\usepackage[english]{babel}

\usepackage[T1]{fontenc}

\usepackage[a4paper,top=3cm,bottom=3cm,left=3cm,right=3cm,marginparwidth=1.75cm]{geometry}

\usepackage{graphicx}
\usepackage{url}
\usepackage[colorlinks,citecolor=violet,urlcolor=violet]{hyperref}
\usepackage{amsfonts,amsmath,amsthm,amssymb,color,tikz-cd}
\usepackage{bbm}
\usepackage{verbatim}
\usepackage{quiver}

\hypersetup{%
  colorlinks=true,
  urlbordercolor={1 1 1},
  linkcolor={violet},
  raiselinks=false
}

\usepackage[shortlabels]{enumitem}

\newcommand{\bb}{\textbf{b}}
\newcommand{\oo}{\textbf{o}}
\newcommand{\ff}{\textbf{f}}
\newcommand{\Z}{\mathbb{Z}}
\newcommand{\Q}{\mathbb{Q}}
\newcommand{\N}{\mathbb{N}}
\newcommand{\e}{\textup{e}}
\newcommand{\T}{\mathcal{T}}
\renewcommand{\S}{\mathcal{S}}
\newcommand{\Pp}{\mathbb{P}}
\newcommand{\E}{\mathbb{E}}
\newcommand{\I}{\mathcal{I}}
\newcommand{\J}{\mathcal{J}}
\newcommand{\diff}{\textup{d}}

\newtheorem{theorem}{Theorem}[section]
\newtheorem{corollary}[theorem]{Corollary}
\newtheorem{lemma}[theorem]{Lemma}

\newtheorem{proposition}[theorem]{Proposition}

\theoremstyle{definition}

\newtheorem{remark}[theorem]{Remark}
\usepackage{mathtools}
\DeclarePairedDelimiter\abs{\lvert}{\rvert}

\newenvironment{nalign}{
    \begin{equation}
    \begin{aligned}
}{
    \end{aligned}
    \end{equation}
    \ignorespacesafterend
}
\begin{document}

\title{Biased random walk on the critical curve of dynamical percolation
}
\author{Assylbek Olzhabayev\footnote{Institute of Science and Technology Austria, assylbek.olzhabayev@ist.ac.at} , Dominik Schmid\footnote{Department of Mathematics, Columbia University, ds4444@columbia.edu}
}

\maketitle

\begin{abstract}
    We study  biased random walks on dynamical percolation in $\Z^d$, which were recently introduced by Andres et al.~in~\cite{andres2023biased}. We provide a second order expansion for the asymptotic speed and show for $d\geq 2$ that the speed of the biased random walk on the critical curve is eventually monotone increasing.  Our methods are based on studying the environment seen from the walker as well as a combination of ergodicity and several couplings arguments.
\end{abstract}

\medskip

\phantom{.} \hspace{0.2cm}\textbf{Keywords:} Dynamical percolation, random walk, environment process \\
\phantom{.} \hspace{0.65cm} \textbf{MSC 2020:}  60K35, 60K37 

\medskip

\section{Introduction}
Random walks on dynamical percolation are studied intensively in recent years~\cite{andres2023biased,gu2024random,gu2024speed,hermon2020comparison,peres2018quenched,peres2020mixing,peres2015random}.
In this paper, we investigate the speed of a biased random walk on dynamical percolation, which was recently introduced by Andres et al.~in~\cite{andres2023biased}. In this model, we consider the integer lattice $\mathbb{Z}^d$ with edge set $E(\Z^d)$ for $d\geq 1$. We define in the following a Markov process $(X_t,\eta_t)_{t\geq 0}$ such that $X_t \in \Z$ and $\eta_t\in \{0,1\}^{E(\mathbb{Z}^d)}$ for all $t\geq 0$. For every edge $e \in E(\mathbb{Z}^d)$, we choose $\eta_0(e)\sim \textup{Ber}(p)$ independently, where $\textup{Ber}(p)$ denotes the Bernoulli-$p$-distribution, and assign an independent Poisson process with rate $\mu>0$. If there is a point in the respective Poisson process at time $t$, we set $\eta_t(e)=1$ with probability $p$, and $\eta_t(e)=0$ with probability $1-p$, independently. We say that an edge $e$ is \textbf{open} at time $t$ if $\eta_t(e)=1$ and \textbf{closed}, otherwise. We refer to $(\eta_t)_{t \geq 0}$ as the \textbf{environment} and  define a continuous-time random walk $(X_t)_{t\geq0}$ in the environment $(\eta_t)_{t\geq0}$ with bias parameter $\lambda>0$ as follows: set $X_0=0$ and assign a rate $1$ Poisson process to the particle. Whenever a clock rings at time $t$ and the random walker is at a site $x$, we choose one of the neighbors $y$ of $x$ with probability $$p(x,x\pm \e_i)=\frac{1}{Z_{\lambda}}\ \text{for} \ i\in\{2,\ldots,d\},$$ 
$$p(x,x+\e_1)=\frac{e^{\lambda}}{Z_{\lambda}}, p(x,x-\e_1)=\frac{e^{-\lambda}}{Z_{\lambda}},$$
whereby $Z_{\lambda}=e^{\lambda}+e^{-\lambda}+2d-2$, and $\pm \e_i$ for all $i\in\{1,\ldots,d\}$ are the unit vectors. If the edge $\{x,y\}$ is open, the walker moves from $x$ to $y$, otherwise the move is suppressed. We will call the process $(X_t, \eta_t)_{t\geq0}$ a $\mathbf{\lambda}$\textbf{-biased random walk on dynamical percolation}, or simply $\boldsymbol\lambda$\textbf{-RWDP}, with parameters $\mu$ and $p$. Note that $(\eta_t)_{t \geq 0 }$ and $(X_t,\eta_t)_{t \geq 0 }$ are Markov processes, while $(X_t)_{t \geq 0}$ is in general not. 
In a broader sense, our objective is to study the position of the walker as a function of $\lambda$. We first recall the main results of Andres et al.~\cite{andres2023biased} on the existence and monotonicity of the speed that establish a framework for our study. Unless otherwise stated, the probability measures are always taking averages over the walk as well as the environment.
\begin{theorem}[Theorems 1.2 and 1.3 in \cite{andres2023biased}]\label{exsp}\label{monotonicity}
    Let $d\geq 1$ and let $(X_t, \eta_t)_{t\geq0}$ be a $\lambda$-RWDP with parameters $\mu>0$ and $p\in (0,1)$. Then for all $\lambda>0$, there exists some  $v(\lambda)=v_{\mu,p}(\lambda)$ such that almost surely $$\lim_{t\to\infty}\frac{X_t}{t}=(v(\lambda),0,\ldots,0).$$
    The function $\lambda\mapsto v(\lambda)$ is strictly positive for all $\lambda>0$ and continuously differentiable.
Moreover, for all $d\geq2$ and $\mu>0$, there exists some $\lambda_0=\lambda_0(\mu,d)$ such that the following two statements hold.
    \begin{enumerate}
        \item The speed $v(\lambda)$ of a $\lambda$-RWDP is strictly increasing for all $\lambda\geq\lambda_0$, provided that $\mu^2>p(1-p).$
        \item The speed $v(\lambda)$ of a $\lambda$-RWDP is strictly decreasing for all $\lambda\geq\lambda_0$, provided that $\mu^2<p(1-p).$
    \end{enumerate}
\end{theorem}
Theorem \ref{exsp} yields a dichotomy in the behavior of the speed of the biased random walk on dynamical percolation. Intuitively, for a sufficiently large update rate $\mu$, the speed behaves as for a biased simple random walk with independently accepted jumps; see also \cite{BZ:InvarianceQuenched} for a treatment under a more general setup.
On the contrary, for sufficiently small update rates, we see the emergence of traps, similar to the biased random walk on static supercritical percolation clusters \cite{berger2003speed,fribergh2014phase,sznitman2003anisotropic}. Andres et al.\ leave in~\cite{andres2023biased} the behavior of the speed of the walker on the \textbf{critical curve} $\mu^2=p(1-p)$ as an open question, which we will address in the sequel as Theorem~\ref{moncrit}.

\subsection{Main results}
We are now ready to state our first main result on the asymptotic speed of the one-dimensional biased random walk on dynamical percolation.
\begin{theorem}\label{mainlemma}
    Let $d=1$, $p\in (0,1)$ and $\mu>0$ arbitrary, but fixed. Then the speed $v(\lambda)$ of a $\lambda$-RWDP satisfies 
\begin{equation}\label{eq:MainLemma1}
v(\lambda)=\frac{\mu p}{1-p+\mu}-\frac{\mu p}{\mu+1-p}C_{\mu,p}e^{-2\lambda}+\mathcal{O}(e^{-4\lambda}) , 
\end{equation} 
 whereby \begin{equation}\label{eq:MainLemma2}
    C_{\mu,p}={\frac {4\,{\mu}^{3}+ 2\left( p+5 \right) {\mu}^{2}+8\,\mu+2\,
 \left( 1-p \right) ^{2}}{ \left( 2\,\mu+1+p \right)  \left( \mu+1-p
 \right)  \left( \mu+p+1 \right) }}, 
\end{equation}
and where the implicit constant depends only on $\mu$ and $p$.
\end{theorem}
Let us remark that the leading order in \eqref{eq:MainLemma1} was previously established in Proposition 4.5 of~\cite{andres2023biased}. As we will see in Section \ref{proof}, the constant $C_{\mu,p}$ is the expected value of some positive random variable, and hence $C_{\mu,p}>0$ for all $\mu>0$, $p\in (0,1)$. Furthermore, a calculations shows that $0<C_{\mu,p}<2$ for all $\mu>0$, $p\in (0,1)$, and that we have for fixed $p\in (0,1)$
\begin{equation}
\lim_{\mu\to\infty}C_{\mu,p}=2 .
\end{equation}
In particular, when $\mu \rightarrow \infty$, the expansion in \eqref{eq:MainLemma1} agrees with the expansion of the speed of a $\lambda$-biased random walk on $\Z$, where every move is accepted independently with probability $p$. Next, we state our main result on the asymptotic expansion of the speed of the biased random walk on dynamical percolation in dimensions $d \geq 2$. 
\begin{theorem}\label{mainlemma2}
    For $d\geq 2$, let $(X_t,\eta_t)_{t\geq 0}$ be a $\lambda$-RWDP on $\mathbb{Z}^d$. For any $p\in (0,1)$, let $\mu=\mu(p)=\sqrt{p(1-p)}$. There exists some constant $\lambda_0$, depending only on $p$ and $d$, such that for all $\lambda \geq \lambda_0$,
    \begin{align*}
        v(\lambda)=\frac{\mu p}{\mu+1-p}-\mathcal{C}_{\mu,p,d}e^{-2\lambda}+\mathcal{O}(e^{-3\lambda}),
    \end{align*}
    whereby $\mathcal{C}_{\mu,p,d}>0$ for all $p\in (0,1)$, and the implicit constant depends only on $p$ and $d$.
\end{theorem}
As a consequence, together with a regularity result on the derivative of the speed in Lemma~\ref{lemma43}, we obtain the following behavior of the speed on the critical curve, resolving Question~1.6 in~\cite{andres2023biased}.
\begin{theorem}\label{moncrit}
    For $d\geq 2$ and $p\in(0,1)$, there exists $\lambda_0=\lambda_0(p,d)$ such that the speed $v_{\mu,p}(\lambda)$ of the $\lambda$-RWDP with parameters $\mu=\sqrt{p(1-p)}$ and $p\in (0,1)$ is strictly monotone increasing for all $\lambda \geq \lambda_0$.
\end{theorem}

\subsection{Related work}
The study of dynamical percolation was initiated in the mathematical literature by H\"aggstr\"om, Peres, and Steif in \cite{olle1997dynamical}, where the existence of an infinite open cluster is investigated. Random walks on dynamical percolation were first investigated by Peres, Stauffer, and Steif \cite{peres2015random} in terms of mixing times and invariance principles, among other properties. A series of papers on the symmetric random walk in dynamical percolation emerged shortly thereafter, studying quenched exit times in \cite{peres2018quenched}, and the mixing time on supercritical dynamical percolation in \cite{peres2020mixing}. In \cite{hermon2020comparison}, Hermon and Sousi establishes a general framework to compare properties between random walks on dynamical percolation and the respective underlying graphs. Other instances of random walks on dynamical percolation include  \cite{sousi2020cutoff} on the cutoff phenomenon for random walks on dynamical Erd\H{o}s--Rényi graphs, or very recently, by Gu et al.\ on speed estimates for random walks in dynamical percolation on non-amenable transitive graphs and the mean squared displacement in Euclidean lattices for small update rates \cite{gu2024random,gu2024speed}. Another direction of research that particularly motivates our work are biased random walks on static percolation clusters, introduced by Barma and Dhar in the physics literature~\cite{barma1983directed}. In the supercritical phase of Bernoulli bond percolation on $\Z^d$ for $d\geq 2$, it was shown independently by Berger et al.\ and Sznitman that the speed of the walk is $0$ for sufficiently large values of the bias \cite{berger2003speed,sznitman2003anisotropic}. Fribergh and Hammond \cite{fribergh2014phase} prove the existence of  a critical parameter that separates the positive speed and the zero speed regime, and show on a subset of the positive speed regime that a central limit theorem holds; see also \cite{bowditch2022biased} for a more recent treatment by Bowditch and Croyden. A strongly related model, the biased random walk on Galton-Watson trees, was studied by Lyons et al.\ in \cite{lyons1996biased}, where similar results were established. There is a large number of papers that concern another models that fall into category of biased random walks in random media, and we refer to \cite{arous2016biased} for an illustrative survey.
\subsection{Overview of ideas}
Our main tool is the environment seen from the walker of  a totally asymmetric biased random walk in dynamical percolation, which we refer to as  TARWDP in the following. In this process, the walker on $\Z$ attempts only jumps to the right.  After every successful jump, we center the environment such that the walker remains at the origin. We show that the environment process started from a Bernoulli-$p$-product measure has the unique stationary distribution, which is asymptotically again a Bernoulli-$p$-product measure. We study the projection to a ball of radius $1$ around the origin, and determine its  invariant measure explicitly.  In order to provide an asymptotic expansion of the speed of the $\lambda$-RWDP with $d=1$ around $\lambda=\infty$, we condition on the event that the walker attempts exactly one jump in the $-e_1$-direction. On this event, we construct a coupling between a $\lambda$-RWDP and a TARWDP. We approximate the probabilities that the edges to the right and to the left of the particle are open at the time of jump attempts in different directions by the stationary distribution of the environment process. Moreover, we show that the two processes can be coupled such that they agree up to a random time shift, whose expectation we compute explicitly. Together with the use of suitable regeneration times from \cite{andres2023biased}, this allows us to convert the time shift into an estimate on the speed. The strategy to establish an asymptotic expansion for the biased random walk on dynamical percolation on $\mathbb{Z}^d$ for $d \geq 2$ is similar. We only consider the case of the critical curve $\mu^2=p(1-p)$, where the term of order $e^{-\lambda}$ in the expansion of the speed vanishes. We decompose the attempted jumps by the walker, depending on whether the jump is in the $\e_1$-direction, in the $-\e_1$ direction, or in the $\pm \e_i$-directions for some $i\geq 2$. Conditioning on the event to see at most one jump in the $-\e_1$ direction or at most two jumps in directions $\pm \e_i$, we use the environment process to estimate the change of speed on the respective events. 
By showing that the derivative has a specific asymptotic expansion, we conclude that the speed of the biased random walk on the critical curve of dynamical percolation is eventually strictly monotone increasing.



\subsection{Organization}
In Section \ref{prelim}, we recall the sequence of regeneration times from \cite{andres2023biased} for biased random walks in dynamical percolation and prove a statement about the regularity of the speed. In Section \ref{environment}, we define the environment process and construct the stationary measure $\Q$.  
We then show that the measure $\Q$ is extremal invariant and asymptotically a Bernoulli-$p$-product measure.  In Section~\ref{proof}, we compute the second order of the asymptotic speed of a one-dimensional biased random walk in dynamical percolation, establishing Theorem~\ref{mainlemma}. In Section~\ref{curve}, we study the speed of the biased random walk on the dynamical percolation on the critical curve $\mu^2=p(1-p)$, and establish that the constant in the second order of the asymptotic expansion is negative. This allows us to conclude Theorems~\ref{mainlemma2} and~\ref{moncrit}.

\subsection{Acknowledgements}
We thank Sebastian Andres and Nina Gantert for fruitful discussions. Parts of this article are based on A.O.'s Bachelor thesis at the University of Bonn. A.O. acknowledges the DAAD program "Stipendienprogramm deutsche Auslandsschulen" for financial support. D.S. is partially funded by the Packard Foundation via Amol Aggarwal's Packard Fellowships for Science and Engineering and by the Simons Foundation via Ivan Corwin's Investigator Award.
\section{Preliminaries}\label{prelim}
\subsection{Infected set and regeneration times}\label{infected}
We start by stating the notion of infected sets, first defined by Peres, Stauffer, and Steif in \cite{peres2015random}. We present the definition as given in \cite[Section 3]{hermon2020comparison} and \cite[Section 2]{andres2023biased}.

We fix an arbitrary enumeration $(k_i)_{i\in\mathbb{N}}$ of the  edges in $E(\mathbb{Z}^d)$. Then for each edge $k_i$, we create an infinite number of copies of this edge denoted by $k_{i,1}$, $k_{i,2}$,\ldots. We now define a process $(I_t)_{t\geq 0}$, where for every $t\geq 0$, $I_t$ is a set containing copies of edges. We refer to $I=(I_t)_{t \geq 0}$ as the \textbf{infected set}. Set $I_0=\emptyset$. Suppose that for some $t\geq 0$, the rate 1 Poisson clock that drives the random walk $(X_t)_{t\geq 0}$ rings, and that the walker examines the edge $k_i$ for some $i\in\mathbb{N}$. If no copy of $k_i$ is contained in $I_{t_-}$, we set $I_t:=I_{t_-}\cup\{k_{i,1}\}$. Otherwise, we add to $I_t$ the copy $k_{i,j}$ of $k_i$ with the smallest index $j$ such that $k_{i,j}\notin I_{t_-}.$

Furthermore, let $(N_t)_{t\geq 0}$ be a Poisson process with time dependent intensity $\mu \abs*{I_t}$. Whenever a clock of this process rings at time $t$, we choose an index uniformly at random from $\{1,\ldots,\abs*{I_t}\}$ and remove the copy of the edge with this index in $I_t$. If the picked copy is of the form $k_{i,1}$ for some $i\in\mathbb{N}$, we set $\eta(k_i)=1$ with probability $p$, and $\eta(k_i)=0$, otherwise. For all edges where no copy is contained in the infected set, we use independent rate $\mu$ Poisson clocks to determine when the state of the respective edge is refreshed.  

Recall that for some $p\in (0,1)$, we start from $\eta_0(k_i)\sim \textup{Ber}(p)$ for every $i\in\mathbb{N}$, $X_0=0$, $I_0=\emptyset$, and note that the evolution of $(\eta_t)_{t\geq 0}$ defined through the infected set has the correct transition rates. 
Define a sequence of random times $(\tau_n)_{n\in\mathbb{N}}$ through $\tau_0:=0$ and 
\begin{align}\label{regtimes}
    \tau_{i+1}:=\inf\{t>\tau_i|I_t=\emptyset \text{ and }I_{t'}\neq\emptyset \text{ for some } t'\in(\tau_i,t)\}.
\end{align}
Observe that these times are regeneration times for $(X_t)_{t\geq0}$, meaning that $(\tau_i-\tau_{i-1})_{i\geq 1}$ are i.i.d. and $(X_{\tau_i}-X_{\tau_{i-1}})_{i\geq1}$ are i.i.d.. The following lemma taken from \cite{hermon2020comparison} uses the fact that the process $(\abs*{I_t})_{t\geq0}$ can be interpreted as a continuous birth-and-death chain on $\mathbb{N}_0$ with transition rates $q(i-1,i)=1$ and $q(i,i-1)=\mu i$ for all $i\in\mathbb{N}$.
\begin{lemma}[Lemma 3.5 in \cite{hermon2020comparison}]\label{lemma35}
    For all $p\in(0,1)$, $\lambda>0$, and $\mu>0$, the increments $(\tau_i-\tau_{i-1})_{i\in\mathbb{N}}$ are i.i.d., have exponential tails, and satisfy $\mathbb{E}[\tau_1]=e^{1/\mu}$.
\end{lemma}
Note that the law of the regeneration times does not depend on $\lambda$ and $p$. For some of our results, we consider the $\lambda$-RWDP with a different initial measure $\mathbb{S}$ that is a product of some arbitrary measure on a finite set $\tilde{I}$ of edges and a Bernoulli-$p$-product measures with parameter $p$ on the rest. In that case, we can still define infected set that starts with $I_0:=\tilde{I}$ for some suitable subset of copies of edges $\tilde{I}$ with the same dynamics. Accordingly, $(\abs*{I_t})_{t\geq0}$ is a Markov chain with same transition rates. Thus, the proof of Lemma \ref{lemma35} that the random times defined by (\ref{regtimes}) have finite expectations can be adopted for this setting. Although $\tau_1$ and $\tau_i-\tau_{i-1}$ will not necessarily have the same distributions anymore, we can still argue that the family $(\tau_i-\tau_{i-1})_{i\geq 1}$ is independent and the family $(\tau_i-\tau_{i-1})_{i\geq 2}$ is i.i.d..
\subsection{Derivative regularity}
Later sections of this paper are devoted to statements concerning the exponential Taylor expansion of the speed $v(\lambda)$, namely Theorem \ref{mainlemma} and Theorem \ref{mainlemma2}. However, those alone are not enough to make the leap to Theorem \ref{moncrit}. The missing piece is the following lemma.
\begin{lemma}\label{lemma43}
    Let $d\geq 1$, and recall from Theorem \ref{monotonicity} that the speed $v(\lambda)$ is continuously differentiable in $\lambda>0$. Then for every $n\in\mathbb{N}$, we get that  $$v'(\lambda)=\sum_{i=1}^n c_ie^{-i\lambda}+\mathcal{O}(e^{-\lambda(n+1)}), $$
    where the implicit constant depends only on $\mu,p$, and $d$. 
\end{lemma}
This lemma is proved with the same tools as Lemma 4.3 in \cite{andres2023biased}. We start by providing the necessary setup. Let $\varepsilon>0$ and $n\in\mathbb{N}$. We couple the $\lambda$-RWDP $(X_t^{\lambda},\eta_t)_{t\geq 0}$ with a $(\lambda+\varepsilon)$-RWDP $(X_t^{\lambda+\varepsilon},\eta_t)_{t\geq 0}$ in the same way as in \cite[Definition~4.9]{andres2023biased}, i.e., we let the environments evolve identically, and we use the same rate $1$ Poisson clocks $(T_i)_{i\in\mathbb{N}}$ to determine when the walkers attempt jumps. Every time a clock rings, we simulate an independent $[0,1]$-uniformly distributed random variable $U$, which will determine the directions of the jump attempts. With other words, we color the points of the Poisson process depending on the value of $U$ at those points.
\begin{enumerate}[(1)]
    \item If $U<(2d-2)Z_{\lambda+\varepsilon}^{-1}$, then we let both walkers attempt a jump into one of the $2d-2$ directions different from $\pm \textup{e}_1$, chosen uniformly at random. We call these points \textbf{bad o-points}.
    \item If $U\in [(2d-2)Z_{\lambda+\varepsilon}^{-1},(2d-2)Z_{\lambda}^{-1}]$, then we let $X^{\lambda}$ attempt a jump into one of the $2d-2$ directions different from $\pm\textup{e}_1$ chosen uniformly at random, while we let $X^{\lambda+\varepsilon}$ attempt a jump in the  $\textup{e}_1$ direction. We call these points \textbf{very bad points}.
    \item If $U\in [(2d-2)Z_{\lambda}^{-1},(2d-2)Z_{\lambda}^{-1}+e^{\lambda+\varepsilon}Z_{\lambda+\varepsilon}^{-1}]$, then we let both walkers attempt a jump into the $-\textup{e}_1$ direction. We call these points \textbf{bad b-points}.
    \item If $U\in [(2d-2)Z_{\lambda}^{-1}+e^{\lambda+\varepsilon}Z_{\lambda+\varepsilon}^{-1},1-e^{\lambda}Z_{\lambda}^{-1}]$, then we let $X^{\lambda}$ attempt a jump in the $-\textup{e}_1$ direction, while we let $X^{\lambda+\varepsilon}$ attempt a jump in the $\textup{e}_1$ direction. With a slight abuse of notation, we call these points \textbf{very bad points} as we will not need to differentiate between cases (2) and (4).
    \item If $U>1-e^{\lambda}Z_{\lambda}^{-1}$, then we let both walkers attempt a jump in the $\textup{e}_1$ direction, and we refer to these points as \textbf{good points}.
\end{enumerate}

Note that good, \bb-bad, \oo-bad, and very bad points occur with intensities $e^{\lambda}Z_{\lambda}^{-1}$, $e^{-\lambda-\varepsilon}Z_{\lambda+\varepsilon}^{-1}$, $(2d-2)Z_{\lambda+\varepsilon}^{-1}$, and $e^{\lambda+\varepsilon}Z_{\lambda+\varepsilon}^{-1}-e^{\lambda}Z_{\lambda}^{-1}$ respectively. All three values are smooth functions of $\lambda$ and $\varepsilon$ that can be expressed by a two-dimensional Taylor approximation. In particular, there exist constants $(c_i)_{i\geq 0}$ such that 
$$q_{vb}:=e^{\lambda+\varepsilon}Z_{\lambda+\varepsilon}^{-1}-e^{\lambda}Z_{\lambda}^{-1}=\sum_{i=0}^n c_i\varepsilon e^{-i\lambda}+\mathcal{O}(\varepsilon e^{-(n+1)\lambda}+\varepsilon^2e^{-\lambda}),$$ with the implicit constant depends only on $d$. With this in mind, we consider the event $A_{\varepsilon}$ that there exists a very bad point on $[0,\tau_1]$. Following the notation of \cite{andres2023biased}, we denote by $\abs{x}_1$ the first coordinate of a vector $x$. To approach the derivative, we notice that 
\begin{equation}\label{eq:Deriv}
\frac{v(\lambda+\varepsilon)-v(\lambda)}{\varepsilon}=\mathbb{E}[\tau_1]^{-1}\mathbb{E}\big[\abs{X_{\tau_1}^{\lambda+\varepsilon}}_1-\abs{X_{\tau_1}^{\lambda}}_1\, \big| \, A_{\varepsilon}\big]\frac{\mathbb{P}[A_{\varepsilon}]}{\varepsilon}.\end{equation}
We will first derive an expansion for $\Pp[A_{\varepsilon}]/\varepsilon$. Denote by $\mathcal{U}_a(t)$ the number of jump attempts of $X^{\lambda}$ on $[0,t]$. Then, we have $\mathbb{P}[A_{\varepsilon}]=\mathbb{E}[1-(1-q_{vb})^{\mathcal{U}_a(\tau_1)}],$. Using that $x\delta(1-\delta x) \leq 1- (1-\delta)^{x} \leq \delta x$ for all $x\in \N$ with $\delta>0$ small enough,  the dominated convergence theorem ensures that
\begin{equation}\label{eq:PTarget}
\begin{split}
    \lim_{\varepsilon\to0}\frac{\mathbb{P}[A_{\varepsilon}]}{\varepsilon}&=\mathbb{E}\left[\lim_{\varepsilon\to0}\frac{1}{\varepsilon}\Big(1-(1-q_{vb})^{\mathcal{U}_a(\tau_1)}\Big)\right]=\lim_{\varepsilon\to 0}\frac{q_{vb}}{\varepsilon}\mathbb{E}[\mathcal{U}_a(\tau_1)]\\&=\sum_{i=0}^n c_ie^{-i\lambda}+\mathcal{O}(e^{-(n+1)\lambda}).
\end{split}
\end{equation}
Our next step is to show that for some constants $(C_i)_{i \in \mathbb{N}}$ \begin{equation}\label{eq:ETarget}
   \lim_{\varepsilon\to 0}\mathbb{E}\big[\abs{X_{\tau_1}^{\lambda+\varepsilon}}_1-\abs{X_{\tau_1}^{\lambda}}_1 \, \big| \, A_{\varepsilon}\big]=\sum_{i=0}^n C_ie^{-i\lambda}+\mathcal{O}(e^{-(n+1)\lambda}). 
\end{equation}
We denote by $\Tilde{A}_{\varepsilon}$ the event that there is a unique very bad point up to time $\tau_1$. Since the occurrence rate of the very bad points goes to $0$ as $\varepsilon\to 0$, conditioning on $A_{\varepsilon}$, the probability of the event that there are at least two very bad points on $[0,\tau_1]$ goes to $0$ as $\varepsilon\to 0$. This is formalized in \cite[Equation (4.18)]{andres2023biased}, which states that 
\begin{align}\label{418}
    \lim_{\varepsilon\to 0}\mathbb{P}[\Tilde{A}_{\varepsilon}|A_{\varepsilon}]=1.
\end{align}
For every $l\in\mathbb{N}$, denote by $V_l$ the event that there is a unique very bad point until $\tau_1$ and this very bad point is the $l^{\text{th}}$ jump attempt by the walkers. By $G_{\mathcal{I},\mathcal{J}}$, denote the event that \bb-bad points have occurred on $[0,\tau_1]$ at times $T_{i_1}$, $T_{i_2}$,\ldots,$T_{i_{s}}$, whereby $\mathcal{I}=(i_1,\ldots,i_{s})$, and \oo-bad points have occurred on $[0,\tau_1]$ at times $T_{j_1}$, $T_{j_2}$,\ldots,$T_{j_{t}}$ whereby $\mathcal{J}=(i_1,\ldots,j_{t})$ for some $s,t\in \N$. Decomposing the event $\{\mathcal{U}_a=k\}\cap V_l$ into disjoint events $\{\mathcal{U}_a=k\}\cap V_l\cap G_{\mathcal{I},\mathcal{J}}$, we obtain
\begin{nalign}\label{fij}
\mathbb{E}[\abs{X_{\tau_1}^{\lambda+\varepsilon}}_1|\mathcal{U}_a(&\tau_1)=k,V_l]\\&=f_0(k,l)+\sum_{\substack{\mathcal{I}\subset\{1,\ldots,l-1,l+1,\ldots,k\}\\\mathcal{J}\subset\{1,\ldots,l-1,l+1,\ldots,k\}\\\mathcal{I}\cap\mathcal{J}=\emptyset}}(f_{\mathcal{I},\mathcal{J}}(k,l)-f_0(k,l))\mathbb{P}[G_{\mathcal{I},\mathcal{J}}|\mathcal{U}_a(\tau_1)=k,V_l]
\end{nalign}
for some functions $f_{\mathcal{I},\mathcal{J}}(k,l)$.
The following lemma states that $f_{\mathcal{I},\mathcal{J}}(k,l)$ from equation \eqref{fij} are independent of $\lambda$ and $\varepsilon$. We omit the proof as it follows mutatis mutandis from \cite[Lemma 4.10]{andres2023biased}.

\begin{lemma}[cf. Lemma 4.10 in \cite{andres2023biased}]\label{lem:AuxTaylor}
    There exist functions $f_{\I,\J}=f_{\I,\J,\mu,p}:\mathbb{N}\times\mathbb{N}\to\mathbb{R}_+$ and $g_{\I,\J}=g_{\I,\J,\mu,p}:\mathbb{N}\times\mathbb{N}\to\mathbb{R}_+$, which do not depend on $\lambda$ or $\varepsilon$, such that 
    $$\mathbb{E}[\abs{X_{\tau_1}^{\lambda+\varepsilon}}_1|\mathcal{U}_a(\tau_1)=k,V_l,G_{\I,\J}]=f_{\I,\J}(k,l)\ and\ \mathbb{E}[\abs{X_{\tau_1}^{\lambda}}_1|\mathcal{U}_a(\tau_1)=k,V_l,G_{\I,\J}]=g_{\I,\J}(k,l).$$
\end{lemma}

We now turn back to the proof of Lemma \ref{lemma43}.
\begin{proof}[Proof of Lemma \ref{lemma43}]
Given the above considerations, we first want to show equation \eqref{fij}. To do so, we have to estimate $\mathbb{P}[G_{\mathcal{I},\mathcal{J}}|\mathcal{U}_a(\tau_1)=k,V_l]$. This quantity can be explicitly calculated given the probabilities that a point of the Poisson process $T$ is good, \oo-bad, \bb-bad, or very bad. The overall expression is a product of some smooth functions of $\lambda$ and $\varepsilon$, which allows us to express it as a two-dimensional Taylor series. More precisely,
\begin{nalign}\label{pgij}
\mathbb{P}[G_{\mathcal{I},\mathcal{J}}|\mathcal{U}_a&(\tau_1)=k,V_l]\\&=\left(e^{\lambda}Z_{\lambda}^{-1}\right)^{k-|\mathcal{I}|-|\mathcal{J}|-1}\left(e^{\lambda+\varepsilon}Z_{\lambda+\varepsilon}^{-1}\right)^{|\mathcal{I}|}\left((2d-2)Z_{\lambda+\varepsilon}^{-1}\right)^{|\mathcal{J}|}\left(1-e^{\lambda+\varepsilon}Z_{\lambda+\varepsilon}^{-1}+e^{\lambda}Z_{\lambda}^{-1}\right)^{1-k}\\&=\sum_{i=|\mathcal{J}|}^n a_i\left(\frac{e^{\lambda}}{2d-2}\right)^{-i}+\mathcal{O}(e^{-(n+1)\lambda}+\varepsilon),
\end{nalign}
whereby $a_i$'s are the coefficients in the exponential Taylor expansion of $e^{k\lambda}Z_{\lambda}^{-k}$, thus
\begin{align}\label{a_i}
    a_i\in\mathcal{O}(k^{i-|\mathcal{J}|}),
\end{align}
with implicit constants that depend only on $d$. Using the obvious bound $\abs{X_{\tau_1}^{\lambda+\varepsilon}}_1\leq\mathcal{U}_a(\tau_1)$, we deduce 
\begin{align}\label{f-f}
    f_{\mathcal{I},\mathcal{J}}(k,l)-f_0(k,l)\in\mathcal{O}(k) , 
\end{align}
whereby the implicit constant depends only on $\mu$ and $d$. Lemma \ref{lem:AuxTaylor} together with equations \eqref{pgij}, \eqref{a_i}, and \eqref{f-f} allow us to write equation \eqref{fij} as
\begin{align}\label{ext1}
    \mathbb{E}[\abs{X_{\tau_1}^{\lambda+\varepsilon}}_1|\mathcal{U}_a(\tau_1)=k,V_l]=f_0(k,l)+\sum_{i=1}^n A_i(k,l)e^{-i\lambda}+\mathcal{O}(k^{2n+3}e^{-(n+1)\lambda}+\varepsilon),
\end{align}
whereby $A_i(k,l)$ is the sum over $\binom{k}{i}$ possible $\mathcal{J}$ with $\abs{\mathcal{J}}=i$, and every term in the summation is in $\mathcal{O}(k^{i+1})$ because of (\ref{a_i}) and (\ref{f-f}). Using the basic estimate $\binom{k}{i}<k^i$, we deduce $A_i(k,l)\in\mathcal{O}(k^{2i+1})$. Similarly, the error terms in equation \eqref{pgij} sum up to $\mathcal{O}(k^{2n+3}e^{-(n+1)\lambda})$, and the implicit constants depend again only on $\mu$ and $p$. From this point onwards, we follow the same arguments as the proof of Lemma~4.3 in~\cite{andres2023biased}. From \cite[Equation (4.20)]{andres2023biased} and \eqref{ext1}, we deduce
\begin{nalign}\label{ext1le}
    \mathbb{E}\big[&\abs{X_{\tau_1}^{\lambda+\varepsilon}}_1 \, \big| \,  \Tilde{A}_{\varepsilon}]\\&=\sum_k\sum_{l\leq k}\Big(f_0(k,l)+\sum_{i=1}^n A_i(k,l)e^{-i\lambda}+\mathcal{O}(k^{2n+3}e^{-(n+1)\lambda}+\varepsilon)\Big)\mathbb{P}[\mathcal{U}_a(\tau_1)=k,V_l\, | \,  \Tilde{A}_{\varepsilon}] , 
\end{nalign}    
and similarly for the process $X^{\lambda}$. From \cite[Equation (4.23)]{andres2023biased}, we get $$\lim_{\varepsilon\to0}\mathbb{P}\big[\mathcal{U}_a(\tau_1)=k,V_l\, \big| \, \Tilde{A}_{\varepsilon}\big]=\frac{\mathbb{P}[\mathcal{U}_a(\tau_1)=k]}{\mathbb{E}[\mathcal{U}_a(\tau_1)]}. $$
Together with \eqref{418}, this allow us to transform \eqref{ext1le} into
\begin{nalign}\label{xle}
    \lim_{\varepsilon\to 0}\mathbb{E}\big[\abs{X_{\tau_1}^{\lambda+\varepsilon}}_1 \, \big| \,  A_{\varepsilon}\big]&=\sum_k\sum_{l\leq k}\Big(f_0(k,l)+\sum_{i=1}^n A_i(k,l)e^{-i\lambda}\Big)\frac{\mathbb{P}[\mathcal{U}_a(\tau_1)=k]}{\mathbb{E}[\mathcal{U}_a(\tau_1)]}\\&+\mathcal{O}\left(e^{-(n+1)\lambda}\sum_k k^{2n+3}\frac{\mathbb{P}[\mathcal{U}_a(\tau_1)=k]}{\mathbb{E}[\mathcal{U}_a(\tau_1)]}\right).
\end{nalign}
To make sure that this is in the desired form, we need to check that the factor of every $e^{-j\lambda}$ is a finite number, i.e., the respective sum converges. Fix $j\leq n$. Then for all $l\leq k$, we have $A_j(k,l)\in\mathcal{O}(k^{2j+1})$.  Hence, there exist $\alpha_j>0$ such that $\sum_{l\leq k}A_j(k,l)<\alpha_jk^{2j+2}$. Summing over all $k\in\mathbb{N}$, we get that the coefficient of $e^{-j\lambda}$ satisfies
$$\sum_k\sum_{l\leq k}A_j(k,l)\frac{\mathbb{P}[\mathcal{U}_a(\tau_1)=k]}{\mathbb{E}[\mathcal{U}_a(\tau_1)]}\leq \alpha_j\sum_k\frac{k^{2j+2}\mathbb{P}[\mathcal{U}_a(\tau_1)=k]}{\mathbb{E}[\mathcal{U}_a(\tau_1)]} .$$ Note that the right-hand side is finite because $\tau_1$ has exponential tails. The same argument applies for the error term, whereby the implicit constant depends only on $\mu$ and $p$. Similarly, we obtain
\begin{nalign}\label{xl}
    \lim_{\varepsilon\to 0}\mathbb{E}\big[\abs{X_{\tau_1}^{\lambda}}_1\, \big|  \, A_{\varepsilon}]&=\sum_{k \geq 1}\sum_{l\leq k}\Big(g_0(k,l)+\sum_{i=1}^n B_i(k,l)e^{-i\lambda}\Big)\frac{\mathbb{P}[\mathcal{U}_a(\tau_1)=k]}{\mathbb{E}[\mathcal{U}_a(\tau_1)]}\\&+\mathcal{O}\left(e^{-(n+1)\lambda}\sum_k k^{2n+3}\frac{\mathbb{P}[\mathcal{U}_a(\tau_1)=k]}{\mathbb{E}[\mathcal{U}_a(\tau_1)]}\right),
\end{nalign}
whereby $B_i(k,l)\in\mathcal{O}(k^{2i+1})$. Subtracting \eqref{xl} from \eqref{xle}, this implies \eqref{eq:ETarget}. Then together with \eqref{eq:PTarget}, recalling the expression for the derivative of the speed in~\eqref{eq:Deriv}, this finishes the proof.
\end{proof}

We conclude this section with the proof of Theorem \ref{moncrit} assuming that Theorem \ref{mainlemma2} holds.
\begin{proof}[Proof of Theorem \ref{moncrit}]
Since the speed is continuously differentiable as a function of $\lambda$ by Theorem \ref{exsp}, for all $s$ large enough 
$$v(2s)-v(s)=\int_s^{2s}v'(t)dt=c_1e^{-s}+2c_2e^{-2s}+\mathcal{O}(e^{-3s}).$$
Taking $s=\lambda$ sufficiently large, we get from Theorem \ref{mainlemma2} that $$c_1=\frac{(2d-2)p}{(1-p+\mu)^2}(\mu^2-p(1-p))=0$$ and $c_2>0$. Taking now the derivative and using that the derivative has an exponential Taylor expansion by Lemma \ref{lemma43}, we conclude that $$v'(\lambda)=2c_2e^{-2\lambda}+\mathcal{O}_{\mu,p}(e^{-3\lambda})>0$$ for all $\lambda$ sufficiently large.
\end{proof}
\subsection{Monotone coupling implies speed disparity}
In this short section, we formulate and prove a proposition that will be used multiple times in Section~\ref{critical} to compare the position of two biased random walks with different initial conditions. 
\begin{proposition}\label{speeddisparity}
    Let $(X_t,\eta_t)_{t\geq 0}$ and $(Y_t,\eta_t')_{t\geq 0}$ be two TARWDPs (corresponding to $\lambda=\infty$ for $\lambda$-RWDP) attempting jumps to the right at rate 1 with environments defined as follows:
    \begin{enumerate}
        \item All $\eta_0(e)$ and $\eta_0'(e)$ for every $e\in E(\Z)$ are independent Bernoulli-distributed random variables;
        \item $0 \leq p_1:=\mathbb{P}[\eta_0(\{0,1\})=1]<\mathbb{P}[\eta_0'(\{0,1\})=1]=:p_2 \leq 1$;
        \item $\mathbb{P}[\eta_0(e)=1]=\mathbb{P}[\eta_0'(e)=1]=:p$ for all edges $e$ except $\{0,1\}$;
        \item Every edge has its state updated at rate $\mu$,  and an updated edge is open with probability $p$ for all edges independently. 
    \end{enumerate}
    Then there exists a coupling $\textbf{P}$ such that $\textbf{P}[X_t\leq Y_t\text{ for all }t>0]=1$ and 
    \begin{align}
       \liminf_{t \rightarrow \infty} \mathbf{E}[X_t-Y_t]=\frac{p_1-p_2}{\mu+1-p}, 
    \end{align} where $\mathbf{E}$ denotes the expectation with respect to $\mathbf{P}$.
\end{proposition}
\begin{proof}
For the first part of the statement, couple the environments $\eta$ and $\eta'$ as follows. For all edges $e$ except $\{0,1\}$, set $\eta_t(e)=\eta_t'(e)$ for all $t\geq 0$. Sample a random variable $U\sim\textup{Unif}[0,1]$, if $U<p_1$, set $\eta_0(\{0,1\})=\eta_0'(\{0,1\})=1$, if $U>p_2$, set $\eta_0(\{0,1\})=\eta_0'(\{0,1\})=0$, otherwise set $\eta_0(\{0,1\})=1-\eta_0'(\{0,1\})=0$. Update the edge $\{0,1\}$ in both environments according to the same Poisson clock and choose the same state at every update. It is easy to verify that the marginal distributions are correct.

We couple the walkers so that their jumps are driven by the same Poisson process. The coupling $\textbf{P}$ of the  environments ensures that 
$$\textbf{P}[\text{the edge }\{X_t,X_t+1\}\text{ is open and the edge }\{Y_t,Y_t+1\}\text{ is closed}]=0
$$
for every $t \geq 0$, meaning that $X$ cannot overtake $Y$ with positive probability.

For the second statement, we use a standard martingale argument to compute the times $S_X$ and $S_Y$ that $X$ and $Y$ need to cross the edge $\{0,1\}$. Note that $\mathbf{P}(S_X \geq  S_Y)=1$, and that
\begin{equation}
     \liminf_{t \rightarrow \infty} \mathbf{E}[X_t-Y_t]= \mathbf{E}[S_Y-S_X]
\end{equation} since the walkers can be coupled after crossing the edge $\{0,1\}$ with respect to the same environment. 
Let $(T_i)_{i\in\mathbb{N}}$ and $(S_i)_{i\in\mathbb{N}}$ denote the rate $\mu$ and rate $1$ Poisson clocks that determine when the edge $\{0,1\}$ is updated and when the walkers attempts to move to the right, respectively. Set $Z_X$ and $Z_Y$ to be the number of the successful jump attempt for $X$ and $Y$, respectively. In particular, we have $S_X=S_{Z_X}$ and $S_Y=S_{Z_Y}$.
    By the memory-less property of the exponential distribution, $\mathbb{P}[Z_X=j|Z_X>j-1]=\frac{\mu p}{\mu+1}$ and $$\mathbb{P}[Z_X=1]=p_1\mathbb{P}[S_1<T_1]+p\mathbb{P}[S_1>T_1]=\frac{p_1+p\mu}{\mu+1},$$ where we use that $S_1\sim \exp(1)$ and $T_1\sim \exp(\mu)$ are independent. Then 
    \begin{equation}
    \begin{split}
            \mathbb{P}[Z_X=j]&=\mathbb{P}[Z_X=j|Z_X>j-1]\left(1-\frac{p_1+p\mu}{\mu+1}\right)\prod_{i=2}^{j-1}\mathbb{P}[Z_X>i|Z_X>i-1]\\
    &=\left(1-\frac{p_1+p\mu}{\mu+1}\right)\left(1-\frac{\mu p}{\mu+1}\right)^{j-2}\frac{\mu p}{\mu+1}
    \end{split}
    \end{equation}
for all $j\geq2$, which allows us to conclude $$\mathbb{E}[Z_X]=\sum_{j=1}^{\infty}j\Pp[Z_X=j]=\frac{\mu+1-p_1}{\mu p}.$$
Let $(\chi_i)_{i\in\mathbb{N}}$ with $\chi_i=S_{i+1}-S_i$ denote the inter-arrival times of the Poisson process $(S_i)_{i\in\mathbb{N}}$. Observe that $Z_X$ is a stopping time with respect to $(\chi_i)_{i\in\mathbb{N}}$, taking the enlarged filtration that contains the process $(T_i)_{i\in\mathbb{N}}$ as well. Note that $(\chi_i)_{i\in\mathbb{N}}$ are i.i.d. Exponential-1-distributed random variables, so by Wald's identity,
$$\mathbb{E}[S_X]=\mathbb{E}[Z_X]\mathbb{E}[\chi_1]=\frac{\mu+1-p_1}{\mu p}.
$$
Analogously, 
$$\E[S_Y]=\frac{\mu+1-p_2}{\mu p}, $$
which finishes the proof.
\end{proof}
\section{The environment process for the one-dimensional RWDP}\label{environment}

In this section, we introduce the environment process as the environment seen from the position of the $\lambda$-RWDP. Throughout this section, we will work with the one-dimensional $\lambda$-RWDP $(X_t,\eta_t)_{t\geq0}$. The goal of this section is to provide the necessary tools for proving Theorem~\ref{mainlemma}.
Our approach is inspired by the study of a tagged particle in exclusion process by Liggett \cite[Part III, Section 4]{liggett1999stochastic}. The original idea dates back to Kesten, Kozlov, and Spitzer \cite{kesten1975limit} in context of random walk in random environment and has since been explored in various contexts, see for example \cite{chen2019limit,gantert2020speed}.

\subsection{Definition and basic properties}
We define the \textbf{environment process} $(\xi_t)_{t\geq0}$ on $\{0,1\}^{\mathbb{Z}}$ through the relation $$\xi_t(\{x,x+1\})=\eta_t(\{X_t+x,X_t+x+1\})$$ for every $\{x,x+1\}\in E(\mathbb{Z})$. In other words, we obtain $\xi_t$ from $\eta_t$ for all $t\geq 0$ by shifting $\mathbb{Z}$ so that the particle stays at the origin. Note that $(\xi_t)_{t\geq 0}$ is again a Markov process and $(X_t,\xi_t)_{t\geq 0}$ completely characterizes a $\lambda$-RWDP. In the following, we will show that for a $\lambda$-RWDP that only attempts jumps in the $\textup{e}_1$-direction at rate $1$, which we will refer to as \textbf{TARWDP} (totally asymmetric random walk in dynamical percolation), the environment process converges to some explicit stationary distribution. We will study properties of this distribution. For convenience, we will write $e_i=\{X_t+i-1,X_t+i\}$ for $i\geq 1$ and $e_i=\{X_t+i,X_t+i+1\}$ for $i\leq -1$. 

First of all, we endow $\{0,1\}^{\mathbb{Z}}$ with a topology. We follow the approach in \cite{gantert2020speed}. For $a,b\in \{0,1\}^{\mathbb{Z}}$, set $$r(a,b)=\min\{n \, | \, a(e_{n})\neq b(e_{n})\ or\ a(e_{-n})\neq b(e_{-n})\}$$ and $d(a,b)=\frac{1}{1+r(a,b)}$. The topology that we use is $\mathcal{T}_d$, induced by the distance function $d$. We call an event $A$ on $\{0,1\}^{\mathbb{Z}}$ \textbf{finitely determined} if there exists an $m\in\mathbb{N}$ such that for any two $\xi,\xi'\in\{0,1\}^{\mathbb{Z}}$ that coincide on $[-m,m]$, either $\xi,\xi'\in A$ or $\xi,\xi'\in A^{\complement}$. The following result is well-known: see, for instance, \cite{lyons1995ergodic}.
\begin{lemma}\label{polish}
    $(\{0,1\}^{\mathbb{Z}},\mathcal{T}_d)$ is a Polish space. Moreover, $\mathcal{B}(\{0,1\}^{\mathbb{Z}},\mathcal{T}_d)$ is generated by finitely determined events.
\end{lemma} 
We make the following basic observation. 
\begin{lemma}\label{Feller}
    The environment process $(\xi_t)_{t\geq 0}$ is a Feller process.
\end{lemma}
\begin{proof}
    By \cite[Theorem 1.5]{liggett1985interacting}, it is enough to prove that for any $\Tilde{\xi}\in\{0,1\}^{\mathbb{Z}}$, 
    \begin{equation}\label{eq:Boxes}
        \lim_{t\to0}\mathbb{E}^{\Tilde{\xi}}[d(\xi_t,\Tilde{\xi})]=0,
    \end{equation}
     where $\mathbb{E}^{\Tilde{\xi}}$ denotes the expectation conditioned on $\xi_0=\Tilde{\xi}$.
    Define the events \begin{align*}&A_0=\{\text{the particle attempts a jump on } [0,t]\}\\&A_m=\{\text{the edge } e_m\text{ was updated on }[0,t]\}\text{ for all } m\in\mathbb{Z}\setminus\{0\}.\end{align*}
    For any $n\in\mathbb{N}$, the probability of the event $A$ that the environment is inside of the ball $[e_{-n},e_n]$ has changed during  the time interval $[0,t]$ goes to $0$ as $t\to 0$, because $A=\bigcup_{i=-n}^n A_i$ and $A_i$ occurs when an exponentially distributed random variable (with parameter $\mu$ if $i\neq 0$ and $1$ otherwise) is smaller than $t$. On the event $A^\complement$, the distance may change by a maximal amount of $\frac{1}{n+1}$. As the maximum distance between two configurations is bounded by $1$, passing $n\to\infty$ implies the desired result \eqref{eq:Boxes}. 
\end{proof}
\subsection{Projection of the environment process on $[-1,1]$}\label{calc}
We begin by heuristically answering a question that is relevant for the proof of Lemma \ref{mainlemma}: at a typical point in time, what is the probability that edges $e_1$ and $e_{-1}$ in the environment process are open? To this end, we consider the projection of the environment process for the $\lambda$-RWDP on $e_1$ and $e_{-1}$. Note that for the special case of a TARWDP, the projection $\xi_t(e_{-1},e_1)_{t\geq 0}$ can be interpreted as a continuous-time Markov chain itself. The reason for it is that the edges which once left the ball with radius $1$ around the particle are almost surely never visited it again, and all edges that enter the ball are open with probability $p$ independently, as the initial distribution is a Bernoulli-$p$-product measure.  We call the $Q$-matrix of this projection simply $Q$, with state space $\{(0,0)$, $(0,1)$, $(1,0)$, $(1,1)\}$, whereby the first number indicates the state of the edge $e_{-1}$ and the second $e_1$. The transition rates are given as in the following diagram.
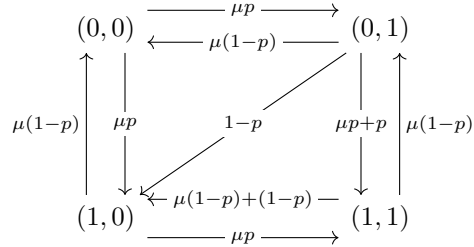
\[\begin{tikzcd}
{(0,0)} \arrow[rrr, "\mu p" description, shift left=3] \arrow[ddd, "\mu p" description, shift left=3] &  &  & {(0,1)} \arrow[lll, "\mu (1-p)" description, shift left=2] \arrow[ddd, "\mu p+p" description, shift right=3] \arrow[lllddd, "1-p" description] \\
                                                                                                      &  &  &                                                                                                                                                \\
                                                                                                      &  &  &                                                                                                                                                \\
{(1,0)} \arrow[uuu, "\mu (1-p)", shift left=3] \arrow[rrr, "\mu p" description, shift right=3]        &  &  & {(1,1)} \arrow[uuu, "\mu (1-p)"', shift right=3] \arrow[lll, "\mu (1-p)+(1-p)" description, shift right=3]                                    
\end{tikzcd}\]
It is apparent that this Markov chain converges to its unique stationary distribution, which is easy to compute explicitly. The $Q$-matrix for the chain is given by $$Q=\begin{pmatrix}
-2\mu p & \mu p & \mu p & 0\\
\mu (1-p) & -\mu & 0 &\mu p\\
\mu (1-p) & 1-p & -\mu-1 &(\mu+1)p\\
0 & (\mu+1)(1-p) & \mu(1-p) &-(1-p)(2\mu+1) , 
\end{pmatrix}$$
so the stationary distribution is the unique solution of $$(x_{(0,0)},\ x_{(1,0)},\ x_{(0,1)},\ x_{(1,1)})Q=0$$ that satisfies $x_{(0,0)}+x_{(1,0)}+x_{(0,1)}+x_{(1,1)}=1$, i.e.,
\begin{align*}
    &{x_{(0,0)}}={\frac {2\,{\mu}^{2}{p}^{2}-\mu{p}^{3}-4\,{\mu}^{2}p+5\,\mu{p}^{2}
+2\,{\mu}^{2}-7\,\mu p+{p}^{2}+3\,\mu-2\,p+1}{2\,{\mu}^{2}-\mu p-{p}^{2}+3\,\mu+1}},\\
&{x_{(1,0)}}=-{\frac { \left( 2\,{\mu}^{2}p-\mu{p}^{2}-2\,{\mu}^{2}+6\,\mu p-5\,\mu
+2\,p-2 \right) p}{2\,{\mu}^{2}-\mu p-{p}^{2}+3\,\mu+1}},\\
&{x_{(0,1)}}=-{\frac {
\mu p \left( 2\,\mu p-{p}^{2}-2\,\mu +2\,p-1 \right) }{2\,{\mu}^{2}-\mu p-{p}^{2}+3
\,\mu +1}},\\
&{x_{(1,1)}}={\frac {{p}^{2}\mu \left( 2\,\mu-p+3 \right) }{2\,{\mu}^{
2}-\mu p-{p}^{2}+3\,\mu +1}}.
\end{align*}
With respect to this stationary distribution, the probability that the edge to the left is open equals 
$$x_{(1,0)}+x_{(1,1)}={\frac {p \left( 2\,{\mu}^{2}-3\,\mu p+5\,\mu-2\,p+2 \right) }{2\,{\mu}^{2}+
 \left( 3-p \right) \mu-{p}^{2}+1}},
$$
and the probability that the edge to the right is open equals 
$$x_{(0,1)}+x_{(1,1)}=\frac{\mu p}{\mu+1-p}.$$
The latter agrees intuitively with \cite[Lemma 4.4]{andres2023biased}: jump attempts occur at rate $1$ and the acceptance probability converges to $\mu p(\mu+1-p)^{-1}$, so $$\lim_{t\to\infty}\frac{X_t}{t}=\frac{\mu p}{\mu+1-p}.$$ This statement is formulated and proved in the next section as Remark \ref{lemma44}.
\subsection{Invariant measure for the environment process of a TARWDP}
Recall from Lemma \ref{Feller} that the environment process $(\xi_t)_{t\geq 0}$ of a TARWDP is a Feller process. We denote its semi-group by $S(t)$ and the distribution of $\xi_0$, i.e., the product of Bernoulli-$p$-distributions on every edge, by $\pi_p$. The goal of this section is to construct a measure $\mathbb{Q}$ on $(\{0,1\}^{\mathbb{Z}},\mathcal{B}(\{0,1\}^{\mathbb{Z}},\mathcal{T}_d))$ such that $\mathbb{Q}$ is an invariant measure for the environment process $(\xi_t)_{t\geq 0}$. Moreover, we aim to show that $\pi_pS(t)\to\mathbb{Q}$ weakly, while the measure $\mathbb{Q}$ is absolutely continuous with respect to $\pi_p$.

Denote by $\mathcal{P}$  the set of all probability measures  on $\{0,1\}^{E(\mathbb{Z})}$, associated with states of edges. We can equip this set with a topology $\mathcal{T}$ induced by weak convergence with respect to the Polish space $(\{0,1\}^{E(\mathbb{Z})},\mathcal{T}_d)$. We have the following result for the environment process, which can be found for example as Proposition 1.8(b) in \cite[Chapter I]{liggett1985interacting} under a more general setting.
\begin{lemma}
    Invariant measures of the environment process are a compact convex subset $\mathcal{I}$ of the topological space $(\mathcal{P},\mathcal{T})$.
\end{lemma}
The following lemma gives an abstract construction of an invariant measure. Recall the infected set $(I_t)_{t\geq 0}$ and regeneration times $(\tau_i)_{i\in\mathbb{N}_0}$ from Section \ref{infected}. 
\begin{lemma}\label{cesaro}
    The probability measure $\mathbb{Q}$ defined by $$\mathbb{Q}[A]=\mathbb{E}[\tau_1]^{-1}\mathbb{E}\left[\int_0^{\tau_1}\pi_p S(t)[A]dt\right]$$ for all finitely determined events $A$ is an invariant measure of the environment process.
\end{lemma}
\begin{proof}
    The proof consists of two steps. First, we prove that we can operate with regeneration times as with deterministic times, arguing similarly to \cite[Proposition 3.2]{andres2023biased}, whereby the original idea dates back to \cite{peres2015random}. Second, we follow the proof of \cite[Chapter I, Proposition 1.8(e)]{liggett1985interacting} to establish that the measure $\mathbb{Q}$ is indeed invariant. Regeneration times have i.i.d. increments and the environment process is Markov, so for any measurable set $A$, 
    $$\mathbb{E}[\tau_1]^{-1}\mathbb{E}\left[\int_0^{\tau_1}\pi_p S(t)[A]dt\right]=\mathbb{E}[\tau_1]^{-1}\mathbb{E}\left[\frac{\int_0^{\tau_n}\pi_p S(t)[A]dt}{n}\right]$$
    for all $n\in\mathbb{N}$, and hence $$\frac{\int_0^{\tau_n}\pi_p S(t)[A]dt}{n}\to\mathbb{E}\left[\int_0^{\tau_1}\pi_p S(t)[A]dt\right]$$ almost surely by the strong law of large numbers for the sequence $(\int_{\tau_{i-1}}^{\tau_{i}}\pi_pS(t)[A]dt)_{i\in\mathbb{N}}$. Therefore,
    $$\frac{\int_0^{\mathbb{E}[\tau_n]}\pi_p S(t)[A]dt}{n}=\frac{\int_0^{\tau_n}\pi_p S(t)[A]dt}{n}+\frac{\int_{\tau_n}^{\mathbb{E}[\tau_n]}\pi_p S(t)[A]dt}{n}\to \mathbb{E}\left[\int_0^{\tau_1}\pi_p S(t)[A]dt\right]$$ as $n\to\infty$, because $\pi_pS(t)[A]\in [0,1]$, so $$\abs*{\frac{\int_{\tau_n}^{\mathbb{E}[\tau_n]}\pi_p S(t)[A]dt}{n}}\leq \frac{|\tau_n-\mathbb{E}[\tau_n]|}{n}\to 0$$ almost surely by the strong law of large numbers for the i.i.d.~sequence $(\tau_i-\tau_{i-1})_{i\in\mathbb{N}}$.
    Now we apply Proposition~1.8(e) from~\cite[Chapter I]{liggett1985interacting} for the sequence $\mathbb{E}[\tau_n]\uparrow\infty$ to obtain that the measure $$\lim_{n\to\infty}\frac{\int_0^{\mathbb{E}[\tau_n]}\pi_p S(t)dt}{\mathbb{E}[\tau_n]}$$ exists, and thus is an invariant measure for $(\xi_t)_{t \geq 0}$.
\end{proof}
\begin{corollary}\label{equivalence}
    The measures $\mathbb{Q}$ and $\pi_p$ are equivalent.
\end{corollary}
\begin{proof}
    Since the environment process is a Feller process, the mapping $t\mapsto\pi_p S(t)[A]$ is right-continuous for every measurable set $A$. Since $\tau_1>0$ a.s., $$\mathbb{Q}[A]=\mathbb{E}[\tau_1]^{-1}\mathbb{E}\left[\int_0^{\tau_1}\pi_p S(t)[A]dt\right]=0$$ if and only if $\pi_p S(t)[A]=0$ almost everywhere on $[0,\tau_1]$, which happens if and only if $\pi_p[A]=0$.
\end{proof}
In the following, our goal is to investigate properties of  the measure $\mathbb{Q}$. We first look at the environment to the right of the particle. 
\begin{lemma}\label{conv_r}
    For any probability measure $\mathbb{S}$ on $\{0,1\}^{E(\mathbb{Z})}$, denote by $\mathbb{P}^{\mu,p,\mathbb{S}}$ the measure corresponding to the TARWDP $(X_t,\xi_t)_{t\geq 0}$ with $\xi_0\sim \mathbb{S}$. Then, $$\lim_{t\to\infty}\mathbb{P}^{\mu,p,\mathbb{S}}[e_i\text{ is open at time }t]=\begin{cases}
        \frac{\mu p}{\mu+1-p},&\textnormal{if}\ i=1\\
        p,\ &\textnormal{if}\ i>1
    \end{cases}.$$
\end{lemma}
Before giving the proof, we record the following consequence of Lemma \ref{conv_r}.
\begin{corollary}\label{right}
    The invariant measure $\mathbb{Q}$ of $(\xi_t)_{t\geq 0}$ satisfies $$\mathbb{Q}[e_1\textnormal{ is open}]=\frac{\mu p}{\mu+1-p}\text{ and }\mathbb{Q}[e_i\textnormal{ is open}]=p\text{ for all }i>1.$$
\end{corollary}
\begin{proof}
    For an invariant measure $\mathbb{Q}$, $$\mathbb{P}^{\mu,p,\mathbb{Q}}[e_i\textnormal{ is open at time }t]=\mathbb{Q}[e_i\text{ is open}],$$ so the claim follows from Lemma \ref{conv_r} with initial measure $\mathbb{Q}$.
\end{proof}
We formulate the following technical lemma that we use in the proof of Lemma \ref{conv_r} as well as multiple times in Sections \ref{proof} and \ref{critical}.
\begin{lemma}\label{gamma_exp}
    For a Gamma-distributed random variable $\T_1\sim \Gamma(n,1)$ with shape $n\in\mathbb{N}$ and scale 1 and an exponentially distributed random variable $\T_2\sim \exp(\mu)$ with rate $\mu>0$, $$\mathbb{P}[\T_1<\T_2]=\frac{1}{(\mu+1)^n}.$$
\end{lemma}
\begin{proof} We give in the following two short independent proofs. For an absolutely continuous random variable $X$, we write $f_X$ and $F_X$ for its density and cumulative distribution, respectively. A simple computation shows that
\begin{align*}
    &\mathbb{P}[\T_1<\T_2]=\int_{-\infty}^{\infty}\int_{-\infty}^{y}f_{\T_1}(x)f_{\T_2}(y)dxdy=\int_{-\infty}^{\infty}f_{\T_2}(y)F_{\T_1}(y)dy
   =\int_{-\infty}^{\infty}\mu e^{-\mu y}\frac{\sum_{i=n}^{\infty}\frac{y^i}{i!}}{e^y}dy\\&=\int_{-\infty}^{\infty}\mu e^{-\mu y}\frac{y^n}{e^y}\sum_{i=0}^{\infty}\frac{y^i}{(n+i)!}dy=\sum_{i=0}^{\infty}\frac{\mu}{(\mu+1)^{n+i+1}}\int_{-\infty}^{\infty}\frac{(\mu+i)^{n+i+1}e^{-(\mu+1)y}y^{n+i}}{(n+i)!}\\&=\sum_{i=0}^{\infty}\frac{\mu}{(\mu+1)^{n+i+1}}=\frac{1}{(\mu+1)^n},
\end{align*}
whereby pulling the sum out of the integral is due to Fubini's theorem. Alternatively, the $\Gamma$-distributed random variable $\T_1$ can be represented as a sum of $n$ independent exponentially distributed random variables $X_1$, $X_2$,\ldots, $X_n$ with rate $1$. Then due to the memory-less property of the exponential distribution,
\begin{align*}
    \mathbb{P}[\T_1<\T_2]=\Pp\left[\sum_{i=1}^nX_i<\T_2\right]=\prod_{i=1}^n\Pp\left[X_i<\T_2\right]=\left(\frac{1}{\mu+1}\right)^n
\end{align*}
as desired.
\end{proof}

\begin{proof}[Proof of Lemma \ref{conv_r}]
For $i=1$, the projection of the TARWDP with $Q$-matrix $Q$ defined in Section~\ref{calc} converges to its stationary distribution in finite time, which implies the desired result for $e_{1}$. Loosely speaking, we need to show that for $i>1$ that edges "get updated faster than they enter some neighbourhood of the particle". Denote by $A_{n,i}$ the event that the edge $\{n+i,n+i+1\}$ was never updated to the time $\T=\inf\{t\geq0 \, | \, X_t=n\}$. We will now make this idea rigorous. Denote the time when the particle makes its $n^{\textup{th}}$ jump attempt by $T_1$, and the time when the edge $\{n+i,n+i+1\}$ gets its first update by $T_2$. As jump attempts and edge updates are driven by Poisson processes with rates 1 and $\mu$, respectively,  we get that  $\T_1\sim \Gamma(n,1)$ and $\T_2\sim \exp(\mu)$. Lemma~\ref{gamma_exp} implies that for $n\in\mathbb{N}$  large enough, $$\mathbb{P}^{\mu,p,\mathbb{S}}[A_{n,i}]\leq \mathbb{P}^{\mu,p,\mathbb{S}}[\T_1<\T_2]=\frac{1}{(\mu+1)^n},$$ because $A_{n,i}=\{\T<\T_1\}\subset \{\T_1<\T_2\}$.
 Consequently, $$\sum_{n=1}^{\infty}\mathbb{P}^{\mu,p,\mathbb{S}}[A_{n,i}]<\infty, $$ so the Borel-Cantelli lemma implies 
 $$\Pp^{\mu,p,\mathbb{S}}[A_{n,i} \text{ for infinitely many } n]=0.$$
 This means that almost surely, there exist some $n_0\in\mathbb{N}$ such that 
 \begin{align}\label{rightenvironment}
     \Pp[e_i \text{ is open at time } t]=p
 \end{align}
 holds for all $\label{eq:conv_r}
     t>\inf\{t\geq 0|X_t=n_0\}$, where the right hand side is a random variable with exponential tails. Hence, \eqref{rightenvironment} follows by taking $t\to\infty$.
\end{proof}
Next, we prove that the measure $\mathbb{Q}$ is extremal invariant, meaning that it belongs to the set of extremal points $\mathcal{I}_e$ of the convex set of all invariant measures $\mathcal{I}$. Corollary 4.14 from \cite[Chapter I]{liggett1985interacting} characterizes $\mathcal{I}_e$ in the following way: $\mu\in\mathcal{I}_e$ holds if and only if every measurable set $A$ such that $\xi_0\in A\Rightarrow \xi_t\in A$ almsot surely holds for every $t>0$ must satisfy $\mu[A]\in\{0,1\}$. We call sets with this property \textbf{invariant sets}.
\begin{lemma}\label{ergodic}
    Recall the environment process $(\xi_t)_{t\geq 0}$ for the TARWDP and let $\mathbb
    {P}^{\mu,p,\xi}$ denote the corresponding measure, where we start deterministically from a configuration $\xi_0=\xi\in\{0,1\}^{E(\Z)}$. Let $A$ be an event such that for almost every $\xi\in A$, and all $t>0$, $$\mathbb
    {P}^{\mu,p,\xi}[\xi_t\in A]=1. $$  Then $\mathbb{Q}[A]\in\{0,1\}$. In particular, $\mathbb{Q}$ is an extremal invariant measure for the environment process.
\end{lemma}
\begin{proof}
    We use the ideas from \cite[Part III, Proposition 4.8]{liggett1999stochastic}. Namely, under the assumptions of the lemma, we first show that $\pi_p[A]=\{0,1\}$ holds. 
    Assume the contrary: there exists an invariant set $A$ with $0<\pi_p[A]<1$. Then, we can find two finitely determined events $B$ and $C$ such that $0<\pi_p[B],\pi_p[C]<1$, $B\subset A$, $C\subset A^{\complement}$, because $A$ and $A^{\complement}$ can be represented as a countable union of finitely determined events.
    
    Take some arbitrary configurations $\xi_1\in B$ and $\xi_2\in C$. Since $B$ and $C$ are finitely determined, there must exist a positive integer $m$ such that in the ball $[e_{-m},e_m]$, the configuration $\xi_2$ differs from any element of $B$. Now we run the environment process beginning with configuration $\xi_1$ inside of $[e_{-m},e_m]$ and a Bernoulli-$p$-product measure outside. By Lemma \ref{lemma35}, $\tau_n<\infty$ almost surely  for all $n\in\mathbb{N}$. We consider the process at points $\tau_1$, $\tau_2$, and so on. At each of these points, the configuration of edges is Bernoulli-$p$-sampled for some $0<p<1$, so $\mathbb{P}[\xi_{\tau_1}\in C]=\pi_p[C]>0$. In particular, after an almost surely finite amount of time the environment process started in $B$ will hit $C$, which contradicts the invariance of $A$.


Consequently, any invariant set is trivial with respect to the measure $\pi_p$. The lemma now follows from the equivalence of $\pi_p$ and $\mathbb{Q}$ established in Corollary~\ref{equivalence}.
\end{proof}
Since $\mathbb{Q}$ is extremal invariant and equivalent to $\pi_p$, we have the following corollary.
\begin{corollary}
    In the topological space of probability measures on $\{0,1\}^{\mathbb{Z}}$, endowed with weak convergence, $\lim_{t\to\infty}\pi_pS(t)=\mathbb{Q}$.
\end{corollary}
\begin{remark}\label{lemma44}
To illustrate the content of this section, we give an alternative proof of  \cite[Lemma~4.4]{andres2023biased} using the environment process and its invariant measure $\mathbb{Q}$. The lemma states that the speed $v=v_{\mu,p}(\infty)$ of a TARWDP $(X_t,\eta_t)_{t\geq 0}$ with parameters $\mu$ and $p$ satisfies 
\begin{align}\label{eq:tarwdpspeed}
    v=\frac{\mu p}{1-p+\mu}.
\end{align}
The proof follows closely the argument developed in \cite[Part III, Section 4]{liggett1999stochastic}. First, note that 
$$v=\lim_{t\to\infty} \frac{1}{t}\left( \int_0^t\psi(\xi_s)ds+M_t \right), $$
where $\psi(\eta)=\eta(e_1)$ and $(M_t)_{t \geq 0}$ is a martingale with respect to its natural filtration as in \cite[Part III, Proposition 4.1]{liggett1999stochastic}. 
By the ergodic theorem \cite[Theorem B50]{liggett1999stochastic} and Corollary \ref{right}, $$\lim_{n\to\infty} \frac{\int_0^t\psi(\xi_s)ds}{t}=\mathbb{E}^{\mu,p,\mathbb{Q}}[\psi(\xi_0)]=\frac{\mu p}{1-p+\mu}\ \text{almost surely.}$$ The martingale $(M_t)_{t \geq 0}$ has stationary and ergodic increments, as $\mathbb{Q}$ is extremal invariant, so $$\lim_{t\to\infty}\frac{M_t}{t}=0$$ almost surely. Summing the last two equations implies  \eqref{eq:tarwdpspeed}.
\end{remark}
\section{Asymptotic expansion of the speed for $d=1$}\label{sec:OneDim}\label{proof}
In this section, our goal is to derive an approximation of the speed of the $\lambda$-RWDP on $\mathbb{Z}$ using the environment process from Section \ref{environment}. We recall the following result from \cite{andres2023biased} that allows use to deduce the speed from the positions of the particle at regeneration times.
\begin{lemma}[Proposition 3.1 in \cite{andres2023biased}]\label{prop31}
    Recall the sequence of regeneration times $(\tau_i)_{i\in\mathbb{N}}$ from \eqref{regtimes}. Then almost surely
    $$\lim_{t\to\infty}\frac{X_t}{t}=v(\lambda)=\frac{\mathbb{E}[X_{\tau_1}]}{\mathbb{E}[\tau_1]}.$$
\end{lemma}
Note that in this lemma, $\tau_1$ can be replaced with $\tau_k$ for any $k\in\mathbb{N}$, because the families $(\tau_i-\tau_{i-1})_{i\in\mathbb{N}}$ and $(X_{\tau_i}-X_{\tau_{i-1}})_{i\in\mathbb{N}}$ are both i.i.d.. Fix some $k\in\mathbb{N}$. Let $\mathcal{P}^{\ff}$ and $\mathcal{P}^{\bb}$ be two independent Poisson processes with intensities $e^{\lambda}Z_{\lambda}^{-1}$ and $e^{-\lambda}Z_{\lambda}^{-1}$ with $Z_{\lambda}=e^{\lambda}+e^{-\lambda}$ that drive the jumps in directions $\e_1$ and $-\e_1$, respectively. We denote by $(N_t^{\bb})_{t\geq 0}$ the counting process of $\mathcal{P}^{\bb}$. We split the probability space into the  following three events:
\begin{align*}
    \mathcal{E}_1&:=\{N_{\tau_k}^{\bb}=1\}\\
    \mathcal{E}_2&:=\{N_{\tau_k}^{\bb}\geq2\}\\
    \mathcal{E}_0&:=\{N_{\tau_k}^{\bb}=0\}.
\end{align*}
In other words, the event
\begin{itemize}[-]
    \item that a jump in the negative direction is attempted exactly once on $[0,\tau_k]$ is $\mathcal{E}_1$,
    \item that a jump in the negative direction is attempted at least twice on $[0,\tau_k]$ is $\mathcal{E}_2$, and
    \item that a jump in the negative direction is not attempted on $[0,\tau_k]$ is $\mathcal{E}_0$.
\end{itemize}
Then by the law of total probability $$\mathbb{E}[X_{\tau_k}]=\mathbb{E}[X_{\tau_k}|\mathcal{E}_0]\mathbb{P}[\mathcal{E}_0]+\mathbb{E}[X_{\tau_k}|\mathcal{E}_1]\mathbb{P}[\mathcal{E}_1]+\mathbb{E}[X_{\tau_k}|\mathcal{E}_2]\mathbb{P}[\mathcal{E}_2].$$ As we will see, the main challenge is to compute $\mathbb{E}[X_{\tau_k}|\mathcal{E}_1]$. To do so, we require some setup. \\

 We construct a coupling between the $\lambda$-RWDP $(X_t,\eta_t)_{t\geq 0}$ and a TARWDP $(Y_t,\Tilde{\eta}_t)_{t\geq 0}$ so that $\eta_t=\Tilde{\eta}_t$ for all $t\geq 0$. Moreover, we let both processes attempt jumps using the same underlying Poisson processes $\mathcal{P}^{\ff}$ and $\mathcal{P}^{\bb}$ to indicate their jump attempts, with the convention that $Y$ attempts a jump in the $\e_1$ direction at the points of $\mathcal{P}^{\bb}$. Let us assume that the event $\mathcal{E}_1$ occurs, and let $\T$ denote the only point of $\mathcal{P}^{\bb}$ on $[0,\tau_k]$. Then at $\T$, the two walkers attempt to jump in different directions, and we define 
 \begin{equation*}
     \begin{split}
         \S_1&:=\inf\{s\geq 0|Y_{\T+s}=Y_{\T_-}+1\} \\
         \S_2&:=\inf\{s\geq 0|X_{\T+s}=X_{\T_-}+1\}
     \end{split}
 \end{equation*}
 as the first times after $\T_-$ to cross an edge to the right.
 Note that $\S_1\leq \S_2$ almost surely and that $Y_{\T+\S_1}$ and $X_{\T+\S_2}$ both have Bernoulli-$p$-product environment to the right, i.e., all edges to the right of $Y_{\T+\S_1}$ at time $\T+\S_1$ and to the right of $X_{\T+\S_2}$ at time $\T+\S_2$ are open independently with probability $p$. 
In order to compare the positions $X_{\tau_k}$ and $Y_{\tau_k}$, it suffices to consider the time difference
 $$\underline{\S}:=\min\{\T+\S_2,\tau_k\}-\min\{\T+\S_1,\tau_k\}, $$
i.e., we compare times for $Y$ and $X$ to reach position $X_{\T_-}+1$, and then run the dynamics $Y$ for an additional time of $\underline{\S}$. \\

This approach comes with two main challenges. First, the random variable $\underline{\S}$ seems rather difficult to approach, so we will instead work with $$\S:=\S_2-\S_1.$$ 
Note that we can substitute $\underline{\S}$ with $\S$ as on the event $\{\T\leq\tau_k\}$, the law of $\S_1$ and $\S_2$ does not depend on $k$, while $\S_1$ and $\S_2$ have exponential tails. In particular, we get that $$\lim_{k\to\infty}\frac{\S}{\underline{\S}}=1$$ almost surely. Second, when running $Y$ for an additional time $\underline{\S}$, its law depends on the location~$\T$. Since $\T \rightarrow \infty $ as $k \rightarrow \infty$, we will argue that we can assume that $Y$ runs for time $\underline{\S}$ starting from the measure $\Q$, which is the invariant measure of the underlying environment seen from the walker. We will formalize both arguments in Sections~\ref{expval} and~\ref{sec:Concluding}, where we complete the proof of Theorem~\ref{mainlemma} using the random variable $\underline{\S}$. \\

We will now focus on computing the expectation of $\S$ when starting both processes from the measure $\Q$ constructed in Lemma \ref{cesaro}.
Denote the environment seen from the walker of $X$ by $(\xi_t^X)_{t\geq0}$. To calculate the expected value of $\S$, we split $\mathcal{E}_1$ into four further events: 
\begin{align*}
    \mathcal{A}&:= \mathcal{E}_1 \cap \{\xi_{\T}(e_{-1})=0,\xi_{\T}(e_1)=1\}\text{ - ''at time $\T$, only $Y$ performed a jump'',}\\\mathcal{B}&:=\mathcal{E}_1 \cap \{\xi_{\T}(e_{-1})=1,\xi_{\T}(e_1)=1\}\text{ - ''at time $\T$, both jumps were successful'',}\\\mathcal{C}&:= \mathcal{E}_1 \cap \{\xi_{\T}(e_{-1})=1,\xi_{\T}(e_1)=0\}\text{ - ''at time $\T$, only $X$ performed a jump'', and}\\\mathcal{D}&:= \mathcal{E}_1 \cap \{\xi_{\T}(e_{-1})=0,\xi_{\T}(e_1)=0\}\text{ - ''at time $\T$, neither walker jumped''.}
\end{align*}
Figure \ref{fig:enter-label} illustrates the positions of the walkers in each of the events $\mathcal{A}$, $\mathcal{B}$, $\mathcal{C}$, and $\mathcal{D}$. Note that when conditioning on one of the events $\mathcal{A},\mathcal{B},\mathcal{C},\mathcal{D}$, the law of $\mathcal{S}$ does not depend on $\Q$ or the Poisson process $\mathcal{P}^{\bb}$, and clearly $\E[\S | \mathcal{D}]=0$. \\

In the following three subsections, we compute the quantities $\E[\S | \mathcal{A}]$, $\E[\S | \mathcal{B}]$ and  $\E[\S | \mathcal{C}]$ in order to derive an expression for $\E[\underline{\S}]$ up to an error of order $k^{-1}+ke^{-2\lambda}$.

\begin{figure}
    \centering
\begin{tikzcd}
	{\mathcal{A}:} & \bullet & \begin{array}{c} \ \\\bullet\\X_{\mathcal{T}} \end{array} & \bullet \\
	{\mathcal{B}:} & \bullet & \bullet & \bullet
	\arrow["closed", dashed, no head, from=1-2, to=1-3]
	\arrow["open", no head, from=1-3, to=1-4]
	\arrow["{Y_{\mathcal{T}}}", curve={height=-18pt}, from=1-3, to=1-4]
	\arrow["open", no head, from=2-2, to=2-3]
	\arrow["{X_{\mathcal{T}}}", curve={height=-18pt}, from=2-3, to=2-2]
	\arrow["open", no head, from=2-3, to=2-4]
	\arrow["{Y_{\mathcal{T}}}", curve={height=-18pt}, from=2-3, to=2-4]
\end{tikzcd}
\qquad \qquad
\begin{tikzcd}
	{\mathcal{C}:} & \bullet & \begin{array}{c} Y_{\mathcal{T}}\\\bullet\\ \  \end{array} & \bullet \\
	{\mathcal{D}:} & \bullet & \begin{array}{c} Y_{\mathcal{T}}\\\bullet\\X_{\mathcal{T}} \end{array} & \bullet
	\arrow["open", no head, from=1-2, to=1-3]
	\arrow["{X_{\mathcal{T}}}", curve={height=-18pt}, from=1-3, to=1-2]
	\arrow["closed", dashed, no head, from=1-3, to=1-4]
	\arrow["closed", dashed, no head, from=2-2, to=2-3]
	\arrow["closed", dashed, no head, from=2-3, to=2-4]
\end{tikzcd}
    
    \caption{\label{fig:enter-label}
Visualization of the events $\mathcal{A}$, $\mathcal{B}$, $\mathcal{C}$, $\mathcal{D}$ used in the coupling of walkers $(X_t)_{t \geq 0}$ and $(Y_t)_{t \geq 0}$, respectively.  
    }
\end{figure}
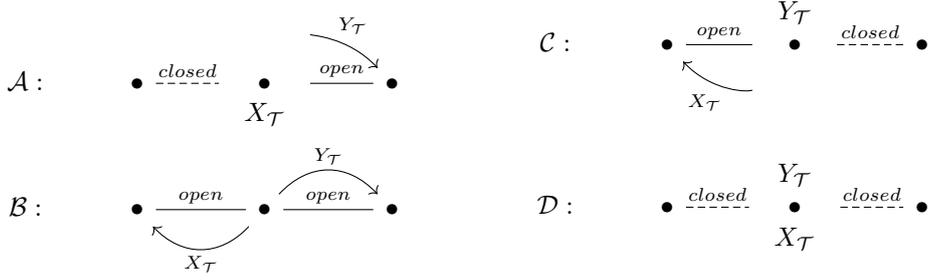

\subsection{At time $\T$, only $Y$ performed a jump}\label{1jump}

We start to compute $\E[\S | \mathcal{A}]$. On the event $\mathcal{A}$, we see that $X_\T=Y_\T-1$ and the edge to the right of  $X_\T$ is open. Moreover, $\S_1=0$. To calculate the expectation of $\S_2$, i.e., the expected time the walker $X$ needs to jump over that edge, let $(T_i)_{i\in\mathbb{N}}$ and $(S_i)_{i\in\mathbb{N}}$ denote the rate $\mu$ and rate $1$ Poisson clocks that determine when this edge is updated and when the random walk attempts to move to the right, respectively. Set $Z$ to be the number of the successful jump attempt. In particular, we have $\S=S_Z$.
    By the memory-less property of the exponential distribution, $\mathbb{P}[Z=j|Z>j-1]=\frac{\mu p}{\mu+1}$ and $$\mathbb{P}[Z=1]=\mathbb{P}[S_1<T_1]+p\mathbb{P}[S_1>T_1]=\frac{1+p\mu}{\mu+1},$$ where we use Lemma \ref{gamma_exp} for $S_1\sim \exp(1)$ and $T_1\sim \exp(\mu)$. Then 
    \begin{equation}\label{eq:RecursionZ}
    \begin{split}
            \mathbb{P}[Z=j]&=\mathbb{P}[Z=j|Z>j-1]\left(1-\frac{1+p\mu}{\mu+1}\right)\prod_{i=2}^{j-1}\mathbb{P}[Z>i|Z>i-1]\\
    &=\left(1-\frac{1+p\mu}{\mu+1}\right)\left(1-\frac{\mu p}{\mu+1}\right)^{j-2}\frac{\mu p}{\mu+1}
    \end{split}
    \end{equation}
for $j\geq2$, which allows us to conclude $\mathbb{E}[Z]=\sum_{j=1}^{\infty}j\Pp[Z=j]={\frac {1}{p}}$.
Let $(\chi_i)_{i\in\mathbb{N}}$ with $\chi_i=S_{i+1}-S_i$ denote the inter-arrival times of the Poisson process $(S_i)_{i\in\mathbb{N}}$. Observe that $Z$ is a stopping time with respect to $(\chi_i)_{i\in\mathbb{N}}$, taking the enlarged filtration that contains the process $(T_i)_{i\in\mathbb{N}}$ as well. Note that $(\chi_i)_{i\in\mathbb{N}}$ are i.i.d. Exponential-1-distributed random variables, so by Wald's identity,
$$\mathbb{E}[\S|\mathcal{A}]=\mathbb{E}[Z]\mathbb{E}[\chi_1]={\frac{1}{p}}.
$$

\subsection{At time $\T$, both walkers performed jumps}\label{eventb}

Next, we aim to compute $\E[\S | \mathcal{B}]$. 
On the event $\mathcal{B}$, we notice that $X_\T=Y_\T-2$, so $X$ needs to compensate 2 edges, and both of them are open at time $\T$. Again, note that $\S_1=0$. Denote by $\S'$ the time $X$ needs to execute the first jump. It follows from the same computations as in Section~\ref{1jump} that $\mathbb{E}[\S']={\frac{1}{p}}.$
In order to compute the time $\S-\S'$ on $\mathcal{B}$, we use a similar argument. Define $(S_i)_{i\in\mathbb{N}}$ to be the Poisson clock of jump attempts, and $(T'_i)_{i\in\mathbb{N}}$ as the Poisson clock for updates of the edge $\{X_\T+1,X_\T+2\}$, both started at time $\T$. Let $Z$ denote the number of the first successful jump attempt, and $$Z'=\inf\{j\in\mathbb{N}:\text{ jump attempt at time $S_{Z+j}$ was successful}\}.$$
As in Section \ref{1jump}, $\mathbb{P}[Z'=j|Z'>j-1]=\frac{\mu p}{\mu+1}$, and $\mathbb{P}[Z'=1]=\mathbb{P}[S_{Z+1}<T'_1]+p\mathbb{P}[S_{Z+1}>T_1]$, whereby
\begin{align*}
    &\mathbb{P}[S_{Z+1}<T'_1]=\sum_{i=1}^{\infty}\mathbb{P}[Z=i]\mathbb{P}[S_{i+1}<T'_1]\\&=\frac{1+p\mu}{(\mu+1)^3}+\sum_{i=2}^{\infty}\frac{1}{(\mu+1)^{i+1}}\left(1-\frac{1+p\mu}{\mu+1}\right)\left(1-\frac{\mu p}{\mu+1}\right)^{i-2}\frac{\mu p}{\mu+1}\\&=\frac{\mu p+p+1}{(\mu +p+1)(\mu+1)^2}.
\end{align*}
Here, we used Lemma \ref{gamma_exp} for $S_{i+1}\sim \Gamma(i+1,1)$ and $T_1\sim\exp(\mu)$.
Now $$\mathbb{P}[Z'=1]=\frac{\mu^3p+(p^2+3p)\mu^2+(p^2+4p)\mu+p+1}{(\mu+p+1)(\mu+1)^2},$$
which allows us to calculate $\mathbb{P}[Z'=j]$ for every $j\in\mathbb{N}$ using the same recursion as in \eqref{eq:RecursionZ}. 
As in Section \ref{1jump}, let $(\chi_i)_{i\in\mathbb{N}}$ with $\chi_i=S_{i+1}-S_i$ denote the inter-arrival times of the Poisson process $(S_i)_{i\in\mathbb{N}}$. Observe that $Z'$ is a stopping time with respect to $(\chi_i)_{i\in\mathbb{N}}$, taking the enlarged filtration that contains the process $(T'_i)_{i\in\mathbb{N}}$ as well. Note that $(\chi_i)_{i\in\mathbb{N}}$ are i.i.d.~Exponential-1-distributed random variables, so by Wald's identity,
\begin{align*}
    \mathbb{E}[\S-\S'|\mathcal{B}]=\mathbb{E}[Z']=\frac{\mu^2+3\mu-p+3}{p(1+\mu)(\mu+p+1)}
\end{align*}
and we conclude by $$\mathbb{E}[\S|\mathcal{B}]=\mathbb{E}[\S']+\mathbb{E}[\S-\S'|\mathcal{B}]=\frac{1}{p}+\frac{\mu^2+3\mu-p+3}{p(1+\mu)(\mu+p+1)}.$$

\subsection{At time $\T$, only $X$ performed a jump}
Next, we aim to compute $\E[\S | \mathcal{C}]$. On the event $\mathcal{C}$, both $\S_1$ and $\S_2$ are almost surely positive. We begin by calculating $\E[\S_2 | \mathcal{C}]$. In other words, we compute the time $X$ needs to jump over an open and then a closed edge, as shown on the diagram below.
\[\begin{tikzcd}
	{X_{\T}} & \bullet & \bullet
	\arrow["{\text{open}}", no head, from=1-1, to=1-2]
	\arrow["{\text{closed}}", no head, from=1-2, to=1-3]
\end{tikzcd}\]
Denote by $\S'$ the time until $X$ jumps across the first edge. The same calculations as in Section \ref{1jump} show that $\mathbb{E}[\S']=\frac{1}{p}$. At time $\T+\S'$, we proceed with a similar argument as in Section \ref{eventb} where the edge $\{X_\T+1,X_\T+2\}$ is assumed to be open at time $\T$. Define by $(S_i)_{i\in\mathbb{N}}$ the Poisson clock of jump attempts, and by $(T'_i)_{i\in\mathbb{N}}$ the Poisson clock for updates of the edge $\{X_\T+1,X_\T+2\}$, both started at time $T$. Again, let $Z$ denote the number of the  first successful jump attempt and define $$Z':=\inf\{j\in\mathbb{N}:\text{ jump attempt at time $S_{Z+j}$ was successful}\}.$$
As in Section \ref{1jump}, we see that $\mathbb{P}[Z'=j|Z'>j-1]=\frac{\mu p}{\mu+1}$ and $\mathbb{P}[Z'=1]=p\mathbb{P}[S_{Z+1}>T_1]$, whereby 
\begin{align*}
    \mathbb{P}[S_{Z+1}<T'_1]=\frac{\mu p+p+1}{(\mu +p+1)(\mu+1)^2}.
\end{align*}
Using the same martingale argument as in Sections \ref{1jump} and \ref{eventb}, we get that $$\mathbb{E}[\S_2-\S' | \mathcal{C}]={\frac {{\mu}^{3}+3\,{\mu}^{2}+3\,\mu+p+1}{p\mu \left( 1+\mu \right)  \left( \mu+p
+1 \right) }} . 
$$
Next, we compute $\mathbb{E}[\S_1]$. Let $(\Tilde{T}_i)_{i\in\mathbb{N}}$ and $(\Tilde{S}_i)_{i\in\mathbb{N}}$ denote the rate $\mu$ and rate 1 Poisson clocks that determine when the edge $\{Y_\T,Y_\T+1\}$ is updated and when the walker $Y$ attempts to move to the right, respectively. Set $\Tilde{Z}$ to be the number of the successful jump attempt. In particular, we have $\S_1=\Tilde{S}_{\Tilde{Z}}$.
    By the memory-less property of the exponential distribution, $\mathbb{P}[\Tilde{Z}=j|\Tilde{Z}>j-1]=\frac{\mu p}{\mu+1}$ and $\mathbb{P}[\Tilde{Z}=1]=p\mathbb{P}[\Tilde{S}_1>\Tilde{T}_1]=\frac{p\mu}{\mu+1}$, using Lemma~\ref{gamma_exp} for $\Tilde{S}_1\sim \exp(1)$ and $\Tilde{T}_1\sim \exp(\mu)$. Then 
    \begin{align*}
    &\mathbb{P}[\Tilde{Z}=j]=\mathbb{P}[\Tilde{Z}=j|\Tilde{Z}>j-1]\left(1-\frac{1+p\mu}{\mu+1}\right)\prod_{i=2}^{j-1}\mathbb{P}[\Tilde{Z}>i|\Tilde{Z}>i-1]\\
    &=\left(1-\frac{p\mu}{\mu+1}\right)\left(1-\frac{\mu p}{\mu+1}\right)^{j-2}\frac{\mu p}{\mu+1},
    \end{align*}
which allows us to conclude that $\mathbb{E}[\S_1 |\mathcal{C}]=\mathbb{E}[\Tilde{Z}]={\frac {1+\mu}{\mu p}}.$ Consequently,
$$\mathbb{E}[\S|\mathcal{C}]={\frac {{\mu}^{2}+2\,\mu-p+1}{p \left( 1+\mu \right)  \left( \mu+p+1 \right) }
} . 
$$

\subsection{Expected time shift of the walker}\label{expval}

As explained at the beginning of this section, we aim to show that $\E[\underline{\S} | \mathcal{E}_1 ]$ is well-approximated by the expression
\begin{align*}
    C_{\mu,p}&:=\mathbb{E}^{\mathbb{Q}}[\S | \mathcal{E}_1]=\mathbb{E}[\S|\mathcal{A}]\mathbb{P}^{\mathbb{Q}}[\mathcal{A}]+\mathbb{E}[\S|\mathcal{B}]\mathbb{P}^{\mathbb{Q}}[\mathcal{B}]+\mathbb{E}[\S|\mathcal{C}]\mathbb{P}^{\mathbb{Q}}[\mathcal{C}]+\mathbb{E}[\S|\mathcal{D}]\mathbb{P}^{\mathbb{Q}}[\mathcal{D}]\\&=\left(-{\frac {\mu
\,p \left( 2\,p\mu-{p}^{2}-2\,\mu+2\,p-1 \right) }{2\,{\mu}^{2}-p\mu-{
p}^{2}+3\,\mu+1}}\right)\mathbb{E}[\S|\mathcal{A}]+\left({\frac {{p}^{2}\mu \left( 2\,\mu-p+3 \right) }{2\,{\mu}^{
2}-\mu p-{p}^{2}+3\,\mu +1}}\right)\mathbb{E}[\S|\mathcal{B}]\\&+\left(-{\frac { \left( 2\,{\mu}^{2}p-\mu{p}^{2}-2\,{\mu}^{2}+6\,\mu p-5\,\mu
+2\,p-2 \right) p}{2\,{\mu}^{2}-\mu p-{p}^{2}+3\,\mu+1}}\right)\mathbb{E}[\mathcal{S}|\mathcal{C}]\\&={\frac {4\,{\mu}^{3}+ 2\left( p+5 \right) {\mu}^{2}+8\,\mu+2\,
 \left( 1-p \right) ^{2}}{ \left( 2\,\mu+1+p \right)  \left( \mu+1-p
 \right)  \left( \mu+p+1 \right) }} , 
\end{align*} where we use the notation $\mathbb{E}^{\Q}$ to stress that the respective environment process is at time $\T$ distributed according to its stationary measure $\Q$ (noting that $\S$ is measurable with respect to $(\xi_{t})_{t \geq\T}$, and $\T$ is a stopping time). 
This approximation is formalized in the following lemma.
\begin{lemma}\label{lem:Approximation} For every $\varepsilon>0$, there exists some $k_0=k_0(\varepsilon,\mu,p)$ such that for all $k\geq k_0$
\begin{equation}
     | \E[\underline{\S} | \mathcal{E}_1] - \E^{\Q}[\S | \mathcal{E}_1] | \leq \varepsilon + \mathcal{O}(e^{-2\lambda}) ,
\end{equation}
    where the implicit constant only depends on $\mu$ and $p$.
\end{lemma}
\begin{proof}
As we cannot give the exact distribution of $\xi_\T(e_{\pm1})$ directly, we will approximate it using the invariant measure $\Q$ of the environment process. We first investigate $\mathbb{E}[\S| \mathcal{E}_1]$.
Note that it suffices to compute the expected value of $\S$ as $$\mathbb{E}[\S| \mathcal{E}_1]=\mathbb{E}[\S|\mathcal{A}]\mathbb{P}[\mathcal{A}]+\mathbb{E}[\S|\mathcal{B}]\mathbb{P}[\mathcal{B}]+\mathbb{E}[\S|\mathcal{C}]\mathbb{P}[\mathcal{C}]+\mathbb{E}[\S|\mathcal{D}]\mathbb{P}[\mathcal{D}],$$ whereby $\mathbb{P}[\mathcal{A}]$, $\mathbb{P}[\mathcal{B}]$, $\mathbb{P}[\mathcal{C}]$, $\mathbb{P}[\mathcal{D}]$ are the probabilities of events $\{(0,1)\}$, $\{(1,1)\}$, $\{(1,0)\}$ and $\{(0,0)\}$ of the Markov chain $Q$, defined in Section \ref{calc}, respectively at time $\T$ (note that $\mathbb{E}[\S|\mathcal{D}]=0$). \\
Recall the environment process $(\xi_t)_{t\geq 0}$ associated with $(X_t,\eta_t)_{t\geq 0}$. The initial measure of the environment process of $(X_t,\eta_t)_{t\geq0}$ is $\pi_p$ (as opposed to $\mathbb{Q}$ for a stationary process), so we have to show that the  actual expected value of $\S$ converges to $C_{\mu,p}$ as $\T\uparrow\infty$ with $k\rightarrow \infty$. Recall that $\T$ is uniformly distributed on $[0,\tau_k]$. For all $\varepsilon^{\prime} \in (0,1)$, we define the $\varepsilon'$-mixing time of the chain $Q$ by $$t_{\textup{mix}}^Q(\varepsilon'):=\inf\left\{t>0\text{ such that } \sup_{\mathcal{M}\subset \{0,1\}^2}\abs{\tilde{\pi}(\mathcal{M})-\tilde{\pi} Q^t(\mathcal{M})}<\varepsilon' \text{ for any initial measure } \tilde{\pi} \right\},$$ Since $t_{\textup{mix}}^Q(\varepsilon')$ only depends on $\varepsilon'$, $\mu$, and $p$, we get that for all $\varepsilon',\delta>0$, there exists some constant $k_0=k_0(\mu,p,\varepsilon',\delta)$ such that for any $k\geq k_0$,  $$\mathbb{P}[\T>t_{\textup{mix}}^{Q}(\varepsilon')|\mathcal{E}_1]>1-\delta.$$ This allows us to decompose $\mathcal{E}_1$ further into the events $\mathcal{E}_1'=\mathcal{E}_1\cap \{\T>t_{\textup{mix}}^Q(\varepsilon')\}$ as well as $\mathcal{E}_1''=\mathcal{E}_1\setminus \{\T>t_{\textup{mix}}^Q(\varepsilon')\}.$ For $k \geq k_0$, note that 
\begin{align}\label{delta}
    \mathbb{P}[\mathcal{E}_1'']<\frac{\delta}{1-\delta}\mathbb{P}[\mathcal{E}_1]
\end{align}
and on the event $\mathcal{E}_1'$, the values of $\mathbb{P}[\mathcal{A}]$, $\mathbb{P}[\mathcal{B}]$, and $\mathbb{P}[\mathcal{C}]$ deviate from the stationary distribution of $Q$ by no more than $\varepsilon'$. In particular, we get that \begin{align}\label{epsilon'}
    \abs*{\frac{\mathbb{E}[\S|\mathcal{E}_1']}{C_{\mu,p}}-1}<\varepsilon'.
\end{align} Since $\varepsilon^{\prime}>0$ can be chosen arbitrarily small, together with \eqref{delta}, this yields that 
\begin{equation}
    | \E[S | \mathcal{E}_1] - \E^{\Q}[\S | \mathcal{E}_1] | \leq \varepsilon + \mathcal{O}(e^{-2\lambda}) 
\end{equation} for all $k$ large enough. 
We now convert this into a comparison between the expectations of $\S$ and $\underline{S}$ on the event $\mathcal{E}_1$, respectively. 
Since $\T\leq\tau_k$, we can bound $\underline{\S}$ by 
\begin{align*}
\underline{\S}&=\min\{\T+\S_2,\tau_k\}-\min\{\T+\S_1,\tau_k\}\leq \tau_k+\S_2-\T\leq \S_2=\S-\S_1 \text{\ and}\\
\underline{\S}&=\min\{\T+\S_2,\tau_k\}-\min\{\T+\S_1,\tau_k\}\geq 0\geq -\S_1=\S-\S_2.
\end{align*}
These two equations can be summarized by $\abs{\underline{\S}-\S}\leq \S_2$.
Moreover, we notice that $\underline{\S}$ and $\S$ do not coincide iff $\tau_k$ is achieved in one of the minima in the definition of $\underline{\S}$, and hence $$\Pp[\underline{\S}\neq \S | \mathcal{E}_1 ]\leq \Pp[\S_2\geq\tau_k-\T | \mathcal{E}_1] . $$
The right-hand side goes to $0$ as $k\to\infty$ noting that  $\tau_k-\T\to\infty$ and $\S_2$ is independent of $k$ and has exponential tails. Thus, since $\E[(\S_2)^2 | \mathcal{E}_1] \leq C$ for some constant $C>0$, depending only on $p,d$ and $\mu$, 
we obtain that
\begin{align*}
    \lim_{k \rightarrow \infty} \abs{\E[\underline{\S} | \mathcal{E}_1]-\E[\S | \mathcal{E}_1]} = 0 , 
\end{align*} allowing us to conclude. 
\end{proof}

\subsection{Concluding the proof of Theorem \ref{mainlemma}  }\label{sec:Concluding}
Using the above approximation of $\E[\underline{S} | \mathcal{E}_1]$,  we will now derive estimates on $\mathbb{E}[Y_{\tau_k-\underline{S}}]$. Recall that  $\mathcal{U}_a(t)$ denotes the number of attempted jumps on $[0,t]$, and note that this is a Poisson distributed random variable with mean $t$. We start with the following observation. 
\begin{lemma}\label{l1conv}
    For a TARWDP $(Y_t)_{t \geq 0}$, we have for positive constants $c,C>0$ and all $t \geq 0$ that
    $$\abs*{\frac{\mathbb{E}[Y_t]}{t}-\frac{\mu p}{\mu+1-p}}\leq Ce^{-ct} . $$ 
\end{lemma}
\begin{proof}
    Conditioned on event that there are exactly $n$ jump attempts on $[0,t]$, the times of those attempts have $n$-dimensional uniform distribution (for instance, \cite[Section 2.1]{daley2003introduction}). Consequently,
    $$\mathbb{E}[Y_t|\mathcal{U}_a(t)=n]=n\int_0^t\mathbb{P}[e_1\text{ open at time }s]\diff s . $$  Now the evolution of the relevant edge states is described by the projection of $Q$ on the second coordinate, which is a Markov chain itself. Denote the projection as $(Z_s)_{s\geq0}$, then $\mathbb{P}[Z_0=1]=p$ and $\lim_{s\to\infty}\mathbb{P}[Z_{s}=1]=\frac{\mu p}{\mu+1-p}$. By the Markov chains convergence theorem -- see for instance,\cite[Theorem 4.9]{levin2017markov} --  there exist positive constants $C$ and $c$ such that $$\abs*{\mathbb{P}[Z_s=1]-\frac{\mu p}{\mu+1-p}}\leq Ce^{-cs}.$$ Integrating this inequality on $[0,t]$ implies $$\mathbb{E}[Y_t|\mathcal{U}_a(t)=n]=n\int_0^t\mathbb{P}[Z_s=1]\diff s\in \left(n\frac{\mu p}{\mu+1-p}-ntCe^{-ct},n\frac{\mu p}{\mu+1-p}+ntCe^{-ct}\right).$$ Using that $\mathcal{U}_a(t)\sim \text{Poi}(t)$, we obtain the desired inequality.
\end{proof}
Next, we approximate $(Y_t)_{t \geq 0}$ along a sequence of random times. 
\begin{lemma}\label{stoppingtimes}
    Consider a TARWDP $(Y_t)_{t \geq 0}$ and a sequence of random times $(\theta_n)_{n\in\mathbb{N}}$ such that there exists a sequence of deterministic intervals $([a_n,b_n])_{n\in\mathbb{N}}$ with 
    \begin{enumerate}
        \item $a_n\to\infty$, $b_n\to\infty$, $|b_n-a_n|$ bounded,
        \item $\mathbb{P}(\theta_n\in[a_n,b_n])\geq 1-e^{-\gamma n}$ for some $\gamma>0$, and
        \item $\limsup_{n \rightarrow \infty} b_n/n < \infty$.
    \end{enumerate}
    Then, we have that for all $n$ large enough $$\abs*{\frac{\mathbb{E}[Y_{\theta_n}]}{\mathbb{E}[\theta_n]}-\frac{\mu p}{\mu+1-p}}\leq Ce^{-cn}.$$
\end{lemma}
\begin{proof}
Denote by $\mathbb{P}[\theta_n\in \diff t]$ the law of $\theta_n$ for all $n\in \N$. Then 
\begin{align*}
\mathbb{E}[X_{\theta_n}]=\int_0^{\infty}\mathbb{E}[X_t]\mathbb{P}[\theta_n\in \diff t]=\int_0^{\infty}\left(\frac{\mu p}{\mu+1-p}+\varepsilon(t)\right)t\mathbb{P}[\theta_n\in \diff t]\\
=\frac{\mu p}{\mu+1-p}\mathbb{E}[\theta_n]+\int_0^t \varepsilon(t)t\mathbb{P}[\theta_n\in \diff t] , 
\end{align*}
where the absolute value of the function $\varepsilon(t)$ decays exponentially in accordance with Lemma \ref{l1conv}. Since $\theta_n$ is concentrated on $[a_n,b_n]$,
\begin{align*}
    \abs*{\int_0^{\infty} \varepsilon(t)t\mathbb{P}[\theta_n\in \diff t]}&\leq\abs*{\int_0^{a_n}\varepsilon(t)t\mathbb{P}[\theta_n\in \diff t]}+\abs*{\int_{a_n}^{b_n}\varepsilon(t)t\mathbb{P}[\theta_n\in \diff t]}+\abs*{\int_{b_n}^{\infty}\varepsilon(t)t\mathbb{P}[\theta_n\in \diff  t]}\\
    &\leq \sup_{t>0}(Cte^{-ct}e^{-\gamma n})+\sup_{[a_n,b_n]}\abs*{\varepsilon(t)}b_n\abs*{b_n-a_n}.
\end{align*}
Again, $\abs*{\varepsilon(t)}\leq Ce^{-ct}$ for all $t \geq 0$, so $$\sup_{[a_n,b_n]}\abs*{\varepsilon(t)b_n(b_n-a_n)}\leq Ce^{-ca_n}b_n(b_n-a_n).$$ Using assumptions $1$ and $3$, we get that  $\mathbb{E}[\theta_n]\in\mathcal{O}(n)$ as well. Thus,
$$\abs*{\mathbb{E}[Y_{\theta_n}]-\frac{\mu p}{\mu+1-p}}\leq Ce^{-cn}\mathbb{E}[\theta_n]\frac{b_n(b_n-a_n)}{\mathbb{E}[\theta_n]}+\sup_{t>0}(Cte^{-ct}e^{-\gamma n}). $$ Since we have that $$\frac{b_n(b_n-a_n)}{\mathbb{E}[\theta_n]}\to 1,$$ this gives the desired result.
\end{proof}
We have now all tools to finish the proof of Theorem~\ref{mainlemma}.
\begin{proof}[Proof of Theorem~\ref{mainlemma}]
Recall that we have that for all $\lambda>0$ and $k\in \N$
\begin{equation}\label{eq:MasterTheorem12}
    v(\lambda) = \frac{\E[X_{\tau_k}]}{\E[\tau_k]} = \frac{1}{\E[\tau_k]} ( \E[X_{\tau_k} | \mathcal{E}_0] \mathbb{P}[\mathcal{E}_0] + \E[X_{\tau_k} | \mathcal{E}_1] \mathbb{P}[\mathcal{E}_1] + \E[X_{\tau_k} | \mathcal{E}_2] \mathbb{P}[\mathcal{E}_2]  ) . 
\end{equation}
Let us start by computing the probability of the events $\mathcal{E}_0$, $\mathcal{E}_1$, $\mathcal{E}_2$. We see that 
\begin{nalign}\label{P(E_1)}
\mathbb{P}[\mathcal{E}_1]&=\sum_{n=0}^{\infty}\mathbb{P}[\mathcal{E}_1|\mathcal{U}_a(\tau_k)=n]\mathbb{P}[\mathcal{U}_a(\tau_k)=n]=\sum_{n=0}^{\infty}n\left(\frac{e^{\lambda}}{e^{\lambda}+e^{-\lambda}}\right)^{n-1}\left(\frac{e^{-\lambda}}{e^{\lambda}+e^{-\lambda}}\right)\mathbb{P}[\mathcal{U}_a(\tau_k)=n]\\&=e^{-2\lambda}\sum_{n=0}^{\infty}n\mathbb{P}[\mathcal{U}_a(\tau_k)=n]+\mathcal{O}(e^{-4\lambda})=ke^{1/\mu}e^{-2\lambda}+\mathcal{O}(e^{-4\lambda})
\end{nalign}
as well as
\begin{nalign}\label{P(E_2)}
    \mathbb{P}[\mathcal{E}_2]&=\sum_{n=0}^{\infty}\mathbb{P}[\mathcal{E}_2|\mathcal{U}_a(\tau_k)=n]\mathbb{P}[\mathcal{U}_a(\tau_k)=n] \\
    &=\sum_{n=0}^{\infty}\frac{n(n-1)}{2}\left(\frac{e^{\lambda}}{e^{\lambda}+e^{-\lambda}}\right)^{n-2}\left(\frac{e^{-\lambda}}{e^{\lambda}+e^{-\lambda}}\right)^2\mathbb{P}[\mathcal{U}_a(\tau_k)=n]\\&=\mathcal{O}(e^{-4\lambda}) . 
\end{nalign}
Consequently, $\mathbb{P}[\mathcal{E}_0]=1-ke^{1/\mu}e^{-2\lambda}+\mathcal{O}(e^{-4\lambda})$.
Using this value, we see that
\begin{align}\label{e0}
    \frac{\mathbb{E}[X_{\tau_k}|\mathcal{E}_0]\mathbb{P}[\mathcal{E}_0]}{\mathbb{E}[\tau_k]}=\frac{\mu p}{\mu+1-p}(1-ke^{\frac{1}{\mu}}e^{-2\lambda})+\mathcal{O}(e^{-4\lambda}),
\end{align}    
as on $\mathcal{E}_0$, $X_{\tau_k}$ has same law as a rate $1$ TARWDP, so \cite[Lemma 4.4]{andres2023biased} and Lemma~\ref{prop31} imply $$\frac{\mathbb{E}[X_{\tau_k}|\mathcal{E}_0]}{\mathbb{E}[\tau_k]}=
\frac{\mu p}{\mu+1-p}.$$ 
Next, recall the events $\mathcal{E}_1'$ and $\mathcal{E}_1''$ from Section \ref{expval} which partition the event $\mathcal{E}_1$ according to the location of the unique bad point. Their probabilities satisfy $\mathbb{P}[\mathcal{E}_1']=(1-\delta_k)\mathbb{P}[\mathcal{E}_1]$ and $\mathbb{P}[\mathcal{E}_1'']=\delta_k\mathbb{P}[\mathcal{E}_1]$ for some constants $(\delta_k)_{k \in \mathbb{N}}$ with $\delta_k \rightarrow 0$ as $k\rightarrow \infty$. Note that random times $(\tau_k-\underline{\S})_{k\in\mathbb{N}}$ satisfy the conditions of Lemma \ref{stoppingtimes}, because $\tau_k$ has exponential tails by Lemma \ref{lemma35} and $\underline{\S}$ is dominated by the waiting time until the TARWDP started from $\mathbb{Q}$ makes its second jump (which is a $\Gamma(2,\mu p(\mu+1-p)^{-1})$-distributed random variable). Applying Lemma \ref{stoppingtimes}, we get that 
\begin{align}
    \mathbb{E}[X_{\tau_k}|\mathcal{E}_1']=\mathbb{E}[Y_{\tau_k-\underline{\S}}|\mathcal{E}_1']=\left(\frac{\mu p}{\mu+1-p}+\varepsilon_1(k)\right)(\mathbb{E}[\tau_k|\mathcal{E}_1']-\mathbb{E}[\underline{\S}|\mathcal{E}_1']),
\end{align} 
whereby $\varepsilon_1(t)$ decays exponentially. Since $\tau_k$ is independent from the directions of the particle, we can simplify the condition by $$\mathbb{E}[\tau_k|\mathcal{E}_1']=\mathbb{E}[\tau_k|\T>t_{\textup{mix}}^Q(\varepsilon'), \mathcal{E}_1],$$ whereby $\T\sim\textup{Unif}([0,\tau_k])$. As $k\to\infty$, $$\Pp[\T>t_{\textup{mix}}^Q(\varepsilon') | \mathcal{E}_1]\to 1,$$ and consequently $$\mathbb{E}[\tau_k|\T>t_{\textup{mix}}^Q(\varepsilon'),  \mathcal{E}_1]\to\mathbb{E}[\tau_k |  \mathcal{E}_1] . $$
We claim now that
\begin{equation}\label{eq:ClaimSize}
  \lim_{k\rightarrow \infty}  \frac{\mathbb{E}[\tau_k |  \mathcal{E}_1]}{\mathbb{E}[\tau_k]} = 1 + \mathcal{O}(e^{-2\lambda}) , 
\end{equation} whereby the implicit constant only depends on $p,d$ and $\mu$. 
To see this, recall that $N^{\bb}_{\tau_k}$ denotes the number of backward jumps until time $\tau_k$. Then we get that
\begin{equation}\label{eq:SizeBias}
    \mathbb{E}[\tau_k |  \mathcal{E}_1] = \mathbb{E}[\tau_k | N^{\bb}_{\tau_k} = 1] = \sum_{x \in \mathbb{N}} x \Pp[ N^{\bb}_{\tau_k} = 1 | \tau_k = x ] \frac{\Pp[\tau_k = x]}{\Pp[N^{\bb}_{\tau_k} = 1]} . 
\end{equation} From the thinning property for the Poisson process $\mathcal{P}^{\ff}+\mathcal{P}^{\bb}$, we see that for all fixed $x\in \mathbb{N}$, 
\begin{align*}
     \Pp[ N^{\bb}_{\tau_k} = 1 |  \tau_k=x ] &= x e^{-2\lambda} + \mathcal{O}(e^{-4\lambda}) ,  \\
     \Pp[N^{\bb}_{\tau_k} = 1] &= \E[\tau_k]e^{-2\lambda} + \mathcal{O}(e^{-4\lambda}) . 
\end{align*} 
Plugging this into \eqref{eq:SizeBias}, we note that
\begin{equation}
    \mathbb{E}[\tau_k |  \mathcal{E}_1] = \frac{\mathbb{E}[\tau^2_k]}{\mathbb{E}[\tau_k]^2} + \mathcal{O}(e^{-2\lambda}) . 
\end{equation}
Since $\textup{Var}[\tau_k]= k \textup{Var}[\tau_1] \leq Ck$ for some constant $C>0$ and all $k\in \mathbb{N}$, this yields the claim \eqref{eq:ClaimSize}. \\
Denoting the error term  $$ \varepsilon_2(k) := \frac{\mathbb{E}[\tau_k|\mathcal{E}_1']}{\mathbb{E}[\tau_k]} - 1 ,$$
Lemma \ref{lem:Approximation} allows us to approximate $\mathbb{E}[\underline{\S}|\mathcal{E}_1']$ by $C_{\mu,p}$ with arbitrarily small error $\varepsilon$.
Hence,
\begin{align}\label{e1'}
    \frac{\mathbb{E}[X_{\tau_k}|\mathcal{E}_1']\mathbb{P}[\mathcal{E}_1']}{\mathbb{E}[\tau_k]}=(ke^{\frac{1}{\mu}}+\varepsilon_2(k)-C_{\mu,p}(1+\varepsilon))\left(\frac{\mu p}{\mu+1-p}+\varepsilon_1(k)\right)(1-\delta)e^{-2\lambda}+\mathcal{O}(e^{-4\lambda})
\end{align}
Last, we observe that by a similar argument
\begin{align}\label{e1''}
    \frac{\mathbb{E}[X_{\tau_1}|\mathcal{E}_1'']\mathbb{P}[\mathcal{E}_1'']}{\mathbb{E}[\tau_k]}&<\frac{\mu p}{\mu+1-p}\delta ke^{1/\mu}e^{-2\lambda}\\\label{e2}
    \frac{\mathbb{E}[X_{\tau_1}|\mathcal{E}_2]\mathbb{P}[\mathcal{E}_2]}{\mathbb{E}[\tau_k]}&=\mathcal{O}(e^{-4\lambda}) . 
\end{align}
Summing (\ref{e0}), (\ref{e1'}), (\ref{e1''}), and (\ref{e2}), we obtain
\begin{align*}
    \vline&\frac{\mathbb{E}[X_{\tau_k}]}{\mathbb{E}[\tau_k]}-\frac{\mu p}{\mu+1-p}+\frac{\mu p}{\mu+1-p}C_{\mu,p}e^{-2\lambda}\vline\\&<\left((1-\delta)\left(k\varepsilon_1(k)+\left(\varepsilon_2(k)-C_{\mu,p}\varepsilon\right)\left(\frac{\mu p}{\mu+1-p}+\varepsilon_1(k)\right)-C_{\mu,p}\varepsilon_1(k)\right)+\frac{\mu p}{\mu+1-p}C_{\mu,p}\delta\right)e^{-2\lambda}\\&+\mathcal{O}(e^{-4\lambda}) .
\end{align*}
Taking $k\to\infty$, we ensure that $\varepsilon_1$, $\varepsilon_2$, $\varepsilon$, and $\delta$ converge to 0 simultaneously. As $\varepsilon_1(k)$ decays exponentially, we get that $k\varepsilon_1(k)\to 0$ as well. Since $\mathbb{E}[X_{\tau_k}]\mathbb{E}[\tau_k]^{-1}$ is constant for any $k\in\mathbb{N}$, we conclude $$v=\lim_{k\to\infty}\frac{\mathbb{E}[X_{\tau_k}]}{\mathbb{E}[\tau_k]}=\frac{\mu p}{\mu+1-p}-\frac{\mu p}{\mu+1-p}C_{\mu,p}e^{-2\lambda}+\mathcal{O}(e^{-4\lambda}),$$
allowing us to conclude. 
\end{proof}

\section{Monotonicity of the speed on the critical curve}\label{curve}
The goal of this section is to complete the missing case $\mu^2=p(1-p)$ in Theorem \ref{monotonicity}. We assume that this relation holds throughout the entire section. Namely, we prove the following result.
\begin{theorem}\label{critical}
    Let $d\geq2$, $p\in (0,1)$, and $\mu>0$ such that $\mu^2=p(1-p)$. Then there exists some constant $\lambda_0=\lambda_0(\mu,d)$ such that the speed $v(\lambda)$ of a $\lambda$-RWDP, defined in Theorem \ref{exsp}, is strictly increasing for all $\lambda\geq\lambda_0$.
\end{theorem}
The overall proof strategy follows a similar approach as Section \ref{sec:OneDim} for the one-dimensional $\lambda$-RWDP. We start with the following lemma. 
\begin{lemma}
    For $d \geq 2$, $\mu>0$, and $p\in(0,1)$ such that $\mu^2=p(1-p)$, the speed $v(\lambda)$ of a $d$-dimensional $\lambda$-RWDP satisfies 
    \begin{align}
        v(\lambda)=\frac{\mu p}{\mu+1-p}-Ce^{-2\lambda}+\mathcal{O}(e^{-3\lambda})
    \end{align}
    for some constant $C>0$, and the implicit constant depending on $\mu$, $p$, and $d$.
\end{lemma}
To prove this lemma, we note that until a fixed time $t$,  with probability $1-\mathcal{O}(e^{-3\lambda})$, where the implicit constant depends  only on $\mu$, $p$, $d$, and $t$, the particle has by time $t$ either
\begin{enumerate}
    \item made exactly one jump attempt in the $-\textbf{e}_1$ direction and all others in the $\textbf{e}_1$ direction or
    \item made at most two jump attempts orthogonal to the $x_1$-axis and all others  in the $\textbf{e}_1$ direction.
\end{enumerate}

We consider these two cases separately. For convenience, we call points of time when a jump  towards $\pm \textbf{e}_i$ for $i\geq 2$ is attempted \textbf{\oo-points}, points of time when a jump towards $-\textbf{e}_1$ is attempted \textbf{\bb-points}, and points of time when a jump towards $\textbf{e}_1$ is attempted \textbf{\ff-points}. Throughout this section, we consider the process up to time $\tau_k$, aiming to take $k\to\infty$ in spirit of Section \ref{proof}. In Section \ref{reducing}, we argue that a \bb-point affects the speed of a $\lambda$-RWDP by the same margin as in the one-dimensional case.  In Section \ref{comparisoncoupling}, we consider the event that there are either one or two \oo-points on $[0,\tau_1]$, proceeding with a strategy similar to \cite{andres2023biased}. In Section \ref{differentcoupling}, we consider the event that the two \oo-jumps are attempted along the same edge, which requires a different coupling.

\subsection{Elimination of backward jumps}\label{reducing}
Consider a $d$-dimensional $\lambda$-RWDP $(\underline{X}_t,\eta_t)_{t\geq 0}$ and denote by $\mathcal{P}^{\ff}$, $\mathcal{P}^{\bb}$, and $\mathcal{P}^{\oo}$ the independent Poisson clocks of the \ff-, \bb-, and \oo-points with rates $e^{\lambda}Z_{\lambda}^{-1}$, $(2d-2)Z_{\lambda}^{-1}$, and $e^{-\lambda}Z_{\lambda}^{-1}$, respectively. We construct a random walker $(X_t)_{t\geq 0}$ on the same environment such that whenever $\mathcal{P}^{\ff}$ or $\mathcal{P}^{\oo}$ rings, it attempts jumps in the same direction as $\underline{X}$. Loosely speaking, $(X_t)_{t\geq 0}$ is obtained from $(\underline{X}_t)_{t\geq 0}$ by eliminating all \bb-points. We will refer to $(X_t,\eta_t)_{t\geq0}$ as the \textbf{reduced $\lambda$-RWDP}. 
\begin{lemma}\label{reducinglemma} Let $k\in \mathbb{N}$. 
    For the $\lambda$-RWDP $(\underline{X}_t,\eta_t)$, we define the event $$\mathcal{E}_{1\bb}:=\{\text{there is exactly one \bb-point on }[0,\tau_k]\}.$$ Then for any $k\in \mathbb{N}$, 
    $$\mathbb{E}[|\underline{X}_{\tau_k}|_1|\mathcal{E}_{1\bb}]=\mathbb{E}[|X_{\tau_k}|_1|\mathcal{E}_{1\bb}]-C_{1\bb}+\mathcal{O}(e^{-\lambda}) , $$ whereby the implicit constant in $\mathcal{O}$ depends  only on $\mu,p,k$ and $C_{1\bb}>0$.
\end{lemma}
\begin{proof}
    Conditioned on $\mathcal{E}_{1\bb}$, RWDP $(\underline{X}_t,\eta_t)_{t\geq 0}$ has an \oo-point on $[0,\tau_1]$ with probability $\mathcal{O}(e^{-\lambda})$, as \bb-points and \oo-points are driven by independent Poisson processes. We claim that it suffices to consider the event $\mathcal{E}_{1\bb}'$ that there has been one \bb-point and zero \oo-points between times $0$ and $\tau_1$. Indeed, since $\Pp[\mathcal{E}_{1\bb}'^\complement|\mathcal{E}_{1\bb}]\in\mathcal{O}(e^{-\lambda})$, we can use the  Cauchy-Schwarz inequality to obtain
    \begin{align*}
        \abs{\E[|X_t|_1\mathbbm{1}_{\mathcal{E}_{1\bb}'^\complement}|\mathcal{E}_{1\bb}]}=\frac{\abs{\E[|X_t|_1\mathbbm{1}_{\mathcal{E}_{1\bb}'^{\complement}}]}}{\Pp[\mathcal{E}_{1\bb}]}\leq \frac{\sqrt{\E[|X_t|_1^2]}\Pp[\mathcal{E}_{1\bb}'^\complement]}{\Pp[\mathcal{E}_{1\bb}]}=\sqrt{\E[|X_t|_1^2]}\Pp[\mathcal{E}_{1\bb}'^\complement|\mathcal{E}_{1\bb}],
    \end{align*}
    whereby $\E[|X_{t}|_1^2]$ for $t=\tau_k$ is of order $k^2$, and $\Pp[\mathcal{E}_{1\bb}'^{\complement}|\mathcal{E}_{1\bb}]\in \mathcal{O}(e^{-\lambda})$. Again,  the implicit constant depends only on $\mu$, $p$, $d$ and $k$.
    On $\mathcal{E}_{1\bb}'$, no edges outside of $x_1$-axis are explored, so we can couple $\underline{X}$ with a one-dimensional $\lambda$-RWDP $\underline{Y}$, time-changed so that it attempts jumps at rate $$\frac{e^{\lambda}+e^{-\lambda}}{e^{\lambda}+e^{-\lambda}+2d-2}.$$
    The existence of a positive constant $C_{1\bb}$ follows now by the same arguments as the asymptotic expansion of $\E[X_{\tau_k} | \mathcal{E}_1]$ in the proof of Theorem \ref{mainlemma}, and we leave the details to the reader.
%
\end{proof}
\subsection{Comparison coupling to a one-dimensional RWDP}\label{comparisoncoupling}

\begin{figure}
\centering
\begin{tikzcd}[column sep=scriptsize]
	\bullet && \bullet && \bullet && \bullet && \bullet && \bullet \\
	\\
	{X_{\T_1-}} && \bullet && {X_{\T_1-}} && \bullet && {X_{\T_1-}} && \bullet \\
	& {\mathcal{E}_{2\oo}'} &&&& {\mathcal{E}_{\uparrow\downarrow}} &&&& {\mathcal{E}_{\uparrow\rightarrow}} \\
	\bullet && \bullet && \bullet && \bullet && \bullet && \bullet \\
	\\
	{X_{\T_1-}} && \bullet && {X_{\T_1-}} && \bullet && {X_{\T_1-}} && \bullet \\
	& {\mathcal{E}_{\uparrow \textbf{u}}} &&&& {\mathcal{E}_\textbf{d}} &&&& {\mathcal{E}_{2\oo}''}
	\arrow[no head, from=1-1, to=1-3]
	\arrow[no head, from=1-5, to=1-7]
	\arrow["{\T_2}"{description}, shift left=3, curve={height=-18pt}, dashed, from=1-5, to=3-5]
	\arrow[no head, from=1-9, to=1-11]
	\arrow["{\T_2}"{description}, shift left=3, curve={height=-18pt}, dashed, from=1-9, to=1-11]
	\arrow[no head, from=3-1, to=1-1]
	\arrow["{\T_1}"{description}, shift left=3, curve={height=-18pt}, dashed, from=3-1, to=1-1]
	\arrow[no head, from=3-1, to=3-3]
	\arrow[no head, from=3-5, to=1-5]
	\arrow["{\T_1}"{description}, shift left=3, curve={height=-18pt}, dashed, from=3-5, to=1-5]
	\arrow[no head, from=3-5, to=3-7]
	\arrow[no head, from=3-9, to=1-9]
	\arrow["{\T_1}"{description}, shift left=3, curve={height=-18pt}, dashed, from=3-9, to=1-9]
	\arrow[no head, from=3-9, to=3-11]
	\arrow[no head, from=5-1, to=5-3]
	\arrow[no head, from=5-5, to=5-7]
	\arrow["{\T_2}"{description}, shift left=3, curve={height=-18pt}, dashed, from=5-5, to=7-5]
	\arrow[no head, from=5-9, to=5-11]
	\arrow["{\T_{\textup{u}}}"{description}, squiggly, no head, from=7-1, to=5-1]
	\arrow["{\T_1}"{description}, shift left=3, curve={height=-18pt}, dashed, from=7-1, to=5-1]
	\arrow[no head, from=7-1, to=7-3]
	\arrow["{\T_\textbf{u}}"{description}, squiggly, no head, from=7-5, to=5-5]
	\arrow["{\T_1}"{description}, shift left=3, curve={height=-18pt}, dashed, from=7-5, to=5-5]
	\arrow[no head, from=7-5, to=7-7]
	\arrow[no head, from=7-9, to=5-9]
	\arrow[no head, from=7-9, to=7-11]
	\arrow[shift right, curve={height=18pt}, dashed, from=7-9, to=7-11]
\end{tikzcd}
\caption{\label{fig:2}Visualization of the events constructed in Section \ref{comparisoncoupling}. Dashed arrows correspond to jumps at the respective times, and squiggly edges to updates.}
\end{figure}
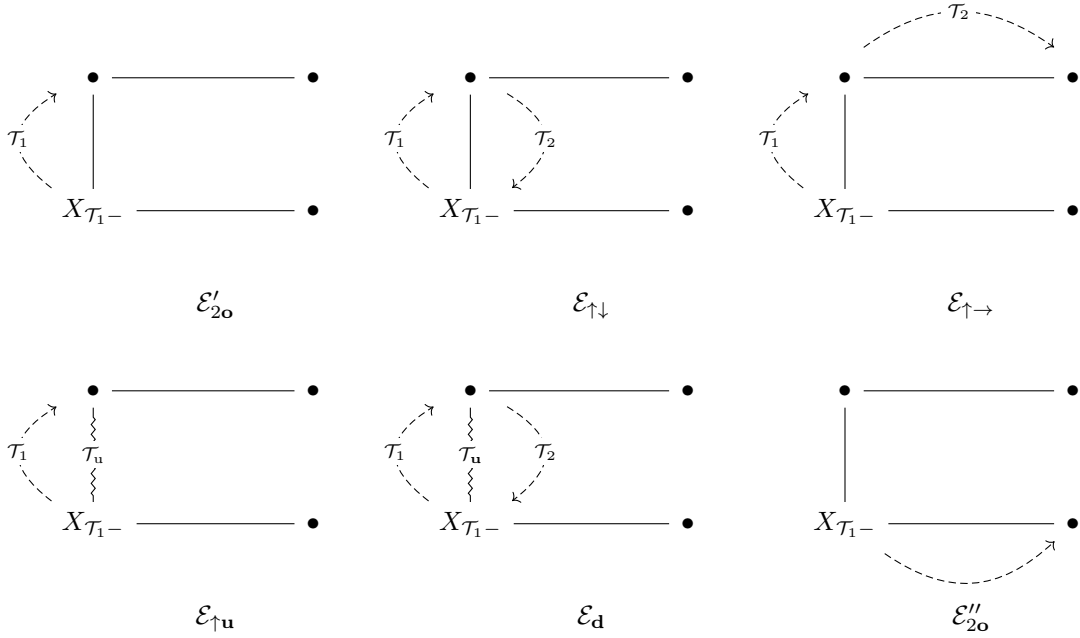
We will now use the coupling as in Definition 4.6 of \cite{andres2023biased} in order to couple the reduced $\lambda$-RWDP with a suitably reduced one-dimensional TARWDP. More precisely, given a $d$-dimensional reduced $\lambda$-RWDP $(X_t,\eta_t)_{t\geq 0}$, consider a one-dimensional TARWDP $(Y_t,\Tilde{\eta}_t)$ with parameters $\Tilde{\mu}=\mu+p(2d-2)Z_{\lambda}^{-1}$ and $p$ that jumps to the right at rate $e^{\lambda}Z_{\lambda}^{-1}=1-(2d-2)Z_{\lambda}^{-1}+\mathcal{O}(e^{-2\lambda})$. Here, we recall that $Z_{\lambda}=e^{\lambda}+(2d-2)+e^{-\lambda}$.
 We will focus on the event $$\mathcal{E}_{\leq 2\oo}:=\{X_t\text{ has at most 2 \oo-points on }[0,\tau_k]\}.$$ This suffices, as by the Cauchy-Schwarz inequality,
 \begin{align*}
     \abs{\E[|X_t|_1\mathbbm{1}_{\mathcal{E}_{\leq 2\oo}^\complement}]}\leq \sqrt{\E[|X_t|_1^2]}\Pp[\mathcal{E}_{\leq 2\oo}^{\complement}],
 \end{align*}
 whereby $\E[|X_t|_1^2]$ is of order $t^2$, and $\Pp[\mathcal{E}_{\leq 2\oo}^{\complement}]=\mathcal{O}(e^{-3\lambda})$. Again, the implicit constant depends only on $\mu,p,d$, and $k$. 
 We couple the two walkers as follows. Recall that we start from product-distributed environments, where every edge is independently open with probability $p$. Walkers attempt jumps to the right according to the same Poisson clock. In addition, $(X_t)_{t \geq 0}$ attempts to jump to a direction orthogonal to the $x_1$-axis at rate $(2d-2)Z_{\lambda}^{-1}$. Until the first \oo-point, we evolve the environments $\eta$ and $\Tilde{\eta}$ together, revealing only the state of the edges along the $x_1$-axis (note that the states of the other edges do not influence the walkers). Suppose that $(X_t)_{t \geq 0}$ successfully performs a jump at the first \oo-point. Then the edge $\{Y_t,Y_t+1\}$ receives an  update at time $t$. Set 
\begin{nalign}\label{eq:o-points}
    \mathcal{E}_{\leq 1\oo}:&=\{X_t \text{ has  at most }  1\ \oo\text{-point  on } [0,\tau_k]\},\\
    \mathcal{E}_{2\oo}:&=\{X_t \text{ has  exactly }  2\ \oo\text{-points  on }  [0,\tau_k]\}.
\end{nalign}
On the event $\mathcal{E}_{2\oo}$, define the first \oo-point by $\T_1$ and denote the second by $\T_2$. Then on $\mathcal{E}_{2\oo}$, denote the event that the first jump at the \oo-point at time $\T_1$ was successful by $\mathcal{E}_{2\oo}'$. 
Moreover, we define the event $\mathcal{E}_{\uparrow\downarrow}\subset \mathcal{E}_{2\oo}'$ that at the time $\T_2$ of the second (successful) orthogonal jump, $X_t$ attempts to cross the same edge $e$ as it has crossed at $\T_1$, and that $e$ was not updated between times $\T_1$ and $\T_2$; see Figure \ref{fig:2} for a visualization of the events. Our goal is to show the following proposition. 
\begin{proposition}\label{monotonecoupling}
   There exists a coupling $\mathbf{P}$ between $X$ and $Y$ such that for any $\delta>0$, there exists some  $k_0:=k_0(\delta,\mu,d,p)$ such that for any $k>k_0$, 
    \begin{enumerate}
        \item $\mathbf{P}[|X_{\tau_k}|_1=Y_{\tau_k}|\mathcal{E}_{\leq 1\oo}]=1$,
        \item $\mathbf{P}[|X_{\tau_k}|_1\leq Y_{\tau_k}|\mathcal{E}^{\prime}_{2\oo}\setminus \mathcal{E}_{\uparrow\downarrow}] = 1$,
        \item $\abs{\mathbf{E}[|X_{\tau_k}|_1-Y_{\tau_k}|\mathcal{E}^{\prime}_{2\oo}\setminus \mathcal{E}_{\uparrow\downarrow}]+C_{2\oo\setminus\uparrow\downarrow}}<\delta+\mathcal{O}(e^{-\lambda}),$ whereby $C_{2\oo\setminus\uparrow\downarrow}>0$. 
    \end{enumerate}
    Here, the implicit constant depends on $\mu,p,d$, and $k$, and we denote by $\mathbf{E}$ expectation under $\mathbf{P}$.
\end{proposition}
\begin{proof}
The first statement in the proposition is  \cite[Proposition 4.1]{andres2023biased}, stating that on the event $\mathcal{E}_{\leq 1\oo}$, the above coupling between $(X_t)_{t \geq 0}$ and  $(Y_t)_{t \geq 0}$ ensures $|X_t|_1=Y_t$ for all $t \leq \tau_k$. 
Since the second and third statement of the proposition only concern large values of $k$, we will in the following assume that the distribution of the state of the edge $e_1$ at time $\T_1$ is close to stationary, namely that it is open with probability $\mu p(\mu+1-p)^{-1}$ up to a arbitrarily small error. 
This statement is a $d$-dimensional version of Lemma \ref{lem:Approximation}, and is proved analogously. 
More precisely, note that on the event $\mathcal{E}_{2\oo}^{\prime}$, the law of $\T_1$ does not depend on the jump process of the walker, and we have $\T_1\to\infty$ as $k\to\infty$. Moreover, as $X$ only attempts  \ff-jumps until $\T_1$, the state of $e_1$ evolves according to the Markov chain $Q$, defined in Section \ref{calc}. Thus, for any $\varepsilon>0$, there exists some  $k_0\in\mathbb{N}$ such that for any $k\geq k_0$, $\T_1$ is with probability at least $1-\delta$ larger than the $\varepsilon$-mixing time of $Q$. Since $\Pp(\mathcal{E}_{\uparrow\downarrow} | \mathcal{E}_{2\oo}^{\prime}) \rightarrow 0$ as $k\rightarrow \infty$, we get that for every $\delta>0$ and $k$ large enough, the edge $e_1$ in the respective environment process $(\eta_t)_{t \geq 0}$ satisfies 
\begin{equation}\label{eq:TVBound}
  \left|  \Pp[\eta_{\mathcal{T}_1}(\{X_{\T_1},X_{\T_1}+\e_1\})=1 | \mathcal{E}_{\leq 2 \oo} \setminus \mathcal{E}_{\uparrow\downarrow}] - \frac{\mu p}{\mu+1-p} \right| \leq  \delta . 
\end{equation}
For the second part of the proposition,
our goal is to show that when a second \oo-point occurs, then on the event $\mathcal{E}^{\prime}_{2\oo}\setminus \mathcal{E}_{\uparrow\downarrow}$, we can choose the coupling $\mathbf{P}$ between $(X_t)_{t \geq 0}$ and $(Y_t)_{t \geq 0}$ such that with high probability, the first coordinate of $X_{\tau_k}$ is  at most $Y_{\tau_k}$. To do so, note that on the event $\mathcal{E}^{\prime}_{2\oo}\setminus \mathcal{E}_{\uparrow\downarrow}$, $Y_{t}$ agrees for all $t \in (\T_1,\tau_k]$ almost surely with the first coordinate of a walker $(\Tilde{X}_{t})_{t \geq 0}$ that starts at time $\T_1$ with $\Tilde{X}_{\T_1}=X_{\T_1}$ in a Bernoulli-$p$ product-distributed environment $(\Tilde{\eta}_t)_{t \geq 0}$, and which attempts all jumps until $\tau_k$ according to the same Poisson process as $(X_t)_{t \geq 0}$. Thus, in view of \eqref{eq:TVBound}, it suffices to show that for any $\delta>0$ and $k$ large enough
 \begin{equation}\label{eq:DominationTarget}
  \mathbf{P}[|X_{\tau_k}|_1 \leq |\Tilde{X}_{\tau_k}|_1|\mathcal{E}^{\prime}_{2\oo}\setminus \mathcal{E}_{\uparrow\downarrow}] \geq 1 - \delta.
 \end{equation}
 Without loss of generality, we assume that the direction of the jump at time $\T_1$ is $\e_2$. We consider the following edges at time $\T_1$:
\begin{align*}
    e_1:=&\{X_{\T_1},X_{\T_1}+\e_1\},\\
    e:=&\{X_{\T_1},X_{\T_1}-\e_2\},\\
    \Tilde{e}_1:=&\{X_{\T_1}-\e_2,X_{\T_1}-\e_2+\e_1\} ; 
\end{align*} 
see Figure \ref{fig:3} for a visualization. 
Then  $\mathcal{E}_{\uparrow\downarrow}$ is the event that at time $\T_2$ of the second orthogonal jump, $X_t$ attempts to cross $e$, and $e$ was not updated between times $\T_1$ and $\T_2$. 
In order to achieve \eqref{eq:DominationTarget}, we partition $\mathcal{E}_{2\oo}^{\prime}\setminus \mathcal{E}_{\uparrow\downarrow}$  further as
\begin{equation}\label{eq:partioningEvents}
    \mathcal{E}_{2\oo}^{\prime}\setminus \mathcal{E}_{\uparrow\downarrow} = \mathcal{E}_{\uparrow\rightarrow} \cup \mathcal{E}_{\uparrow \textbf{u}} ;
\end{equation}
 see again Figure \ref{fig:2} for an illustration. Then  $\mathcal{E}_{\uparrow\rightarrow}$ denotes the event that the second orthogonal jump does not examine the edge $e$. Since on the event $\mathcal{E}_{\uparrow\rightarrow}$, all edges examined by the walkers are initially open independently with probability $p$, we see that
\begin{equation*}
\mathbf{P}[X_{\tau_k} = \Tilde{X}_{\tau_k}| \mathcal{E}_{\uparrow\rightarrow}]=1 .
\end{equation*}
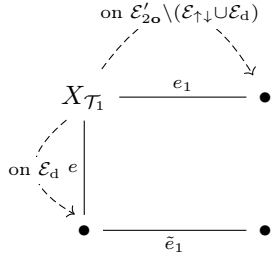
\begin{figure}
\centering
\begin{tikzcd}
	{X_{\T_1}} && \bullet \\
	\\
	\bullet && \bullet
	\arrow["{e_1}", no head, from=1-1, to=1-3]
	\arrow["{\text{on }\mathcal{E}_{2\oo}^{\prime}\setminus (\mathcal{E}_{\uparrow\downarrow}\cup\mathcal{E}_{\textup{d}})}"{description}, shift left, curve={height=-30pt}, dashed, from=1-1, to=1-3]
	\arrow["e"', no head, from=1-1, to=3-1]
	\arrow["{\text{on }\mathcal{E}_{\textup{d}}}"{description}, shift left, curve={height=18pt}, dashed, from=1-1, to=3-1]
	\arrow["{\Tilde{e}_1}"', no head, from=3-1, to=3-3]
\end{tikzcd}
\caption{\label{fig:3}The position of the walker $X$ at time $\T_1$. Dashed arrows show the direction of the next accepted jump.}
\end{figure}
Now consider the event $\mathcal{E}_{\uparrow \textbf{u}}$ that there has been an update of $e$ between $\T_1$ and $\T_2$, and that the edge $e$ is examined at time $\T_2$ by the walkers. Then on the event $\mathcal{E}_{\uparrow \textbf{u}}$, we couple $X$ and $\Tilde{X}$ so that both walkers attempt jumps according to the same Poisson clock and in the same direction. At the same time, we couple the environments after time $\T_1$ so that all edges on the right half-plane with respect to $X_{\T_1}$, except $e$ and $\Tilde{e}_1$, are in the same state and get updated synchronously. Note that for all $k$ large enough such that $\delta< p-\mu p(\mu+1-p)^{-1}$ in  \eqref{eq:TVBound}, we can couple the state of the edge $\Tilde{e}_1$ in $(\eta_t)_{t \geq 0}$ and $(\Tilde{\eta}_t)_{t \geq 0}$, respectively, such that almost surely 
\begin{equation}
 \Tilde{\eta}_{\T_1}( \Tilde{e}_1) \geq\eta_{\T_1}(\Tilde{e}_1) . 
\end{equation}
Now at time $\T_{\textbf{u}}< \T_{2}$ of the first update of $e$, we choose the same state of $e$ in both environments. 
Note that the only event when the walkers $X$ and $\Tilde{X}$ may disagree is the sub-event $\mathcal{E}_{\textbf{d}}$ of $\mathcal{E}_{\uparrow \textbf{u}}$ when the edge $e$ is open at time $\T_2$. On this event, $X$ and $\Tilde{X}$ are only going to attempt \ff-jumps in $[\T_2,\tau_k]$, and their environments only differ at the edge $\Tilde{e}_1$, whereby $X$ has at time $\T_2$ a smaller probability to see an open edge than $\Tilde{X}$. A monotone coupling between two TARWDPs with the same jump rates allows us to couple $X$ and $\Tilde{X}$ after time $\T_2$ so that $X_t\leq\Tilde{X}_t$ for any $t\in[\T_2,\tau_k].$
%
%
This finishes the proof of the second part of the proposition. 

For the third statement, note that on the event $\mathcal{E}_{\textbf{d}}$ with $\Pp(\mathcal{E}_{\textbf{d}})=c e^{-2\lambda}+\mathcal{O}(e^{-3\lambda})$ for some constant $c>0$, the environments of $X$ and $\Tilde{X}$ along the $x_1$-axis differ at exactly one edge, which, together with the strong Markov property at time $\T_2$, allows us to use Proposition~\ref{speeddisparity} to get that the difference in expectations has the desired form.
\end{proof}
Next, we consider the event $$\mathcal{E}_{2\oo}'':=\mathcal{E}_{2\oo}\setminus \mathcal{E}_{2\oo}' , $$ where we have two orthogonal jumps until time $\tau_k$, but the first $\oo$-jump is not successful. The following lemma is similar to Proposition \ref{monotonecoupling}. 

\begin{lemma}\label{lem:Monotone2}
    There exists a coupling $\mathbf{P}$ between $X$ and $Y$ such that for every $\delta>0$, we find a constant $k_0:=k_0(\delta,\mu,p,d)$ such that for all $k \geq k_0$,
    \begin{enumerate}
        \item $\mathbf{P}[|X_{\tau_k}|_1\leq Y_{\tau_k}|\mathcal{E}_{2\oo}'']>1-\delta$; 
        \item $\abs{\mathbf{E}[|X_{\tau_k}|_1-Y_{\tau_k}|\mathcal{E}_{2\oo}'']+C_{2\oo''}}<\delta+\mathcal{O}(e^{-\lambda}),$ whereby $C_{2\oo''}>0$, and the implicit constant only depends on $p,\mu,d$ and $k$, and $\mathbf{E}$ denotes the expectation under $\mathbf{P}$.
    \end{enumerate}
\end{lemma}

\begin{proof}
As in \eqref{eq:TVBound}, choosing $k$ large enough, we assume that the state of the edge $e_1$ is at time $\T_1$ distributed according to the stationary distribution, up to an error $\delta>0$. Then on the event $\mathcal{E}_{2\oo}''$, we couple $X$ and $Y$ directly as follows. At time $\T_1$, we proceed with the coupling from Definition~4.6 of \cite{andres2023biased}, discussed in the previous section. 
Note that $Y$ does not receive an extra update of the edge to the right at time $\T_1$ as $X$ has not jumped. When $X$ receives the second \oo-point at time $\T_2$ on the event $\mathcal{E}_{2\oo}''$, let $\underline{p}$ denote the acceptance probability of the second orthogonal jump, and note that  $\underline{p}<p$. At this point, let $U$ be uniformly distributed on $[0,1]$. The edge to the right of $Y_{\T_2}$ gets an extra update if and only if $U<p$, and $X$ successfully jumps to another coordinate if and only if $U<\underline{p}$. The probability that $Y_{\T_2}$ has an open edge to the right is $$p^2+(1-p)\frac{\mu p}{\mu+1-p},$$ the probability that $X_{\T_2}$ has an open edge to the right is $$\underline{p}p+(1-\underline{p})\frac{\mu p}{\mu+1-p},$$ which is smaller.  This allows us to conclude the first part. Since the distributions of the environments to the right of $X_{\T_2}$ and $Y_{\T_2}$ differ only for the edge $e_1$, we use the strong Markov property for the processes $(X,\eta)$ and $(Y,\Tilde{\eta})$,  and Proposition~\ref{speeddisparity} to obtain the second statement. 
\end{proof}

\subsection{Alternative coupling for a double jump}\label{differentcoupling}
In this subsection, we consider the remaining event $\mathcal{E}_{\uparrow\downarrow}$. We compare $(X_t)_{t\geq 0}$ with a variation of a one-dimensional TARWDP $(\overline{X}_t)_{t\geq 0}$ that corresponds to the projection of $(X_t)_{t \geq 0}$ onto the $x$-axis. Formally, on the event $\mathcal{E}_{\uparrow\downarrow}$, we define $(\overline{X}_t)_{t\geq 0}$ to perform the same moves as $(X_t)_{t \geq 0}$ until  
the first \oo-point $\T_1$. Then $\overline{X}$ does not attempt any jumps until $\T_2$, whereas the environment continues to evolve as for the process $X$. At time $\T_2$, $X$ and $\overline{X}$ are coupled again, and from that point onwards, continue to evolve together. Clearly, we have that on the event $\mathcal{E}_{\uparrow\downarrow}$, the first coordinate of the processes $X$ and $\overline{X}$ agrees at time $\tau_k$ for every $k\in\N$. We construct now on the event $\mathcal{E}_{\uparrow\downarrow}$ a coupling $\mathbf{P}$ between $\overline{X}$ and a TARWDP $\Tilde{Y}$, which jumps at rate $e^{\lambda}Z_{\lambda}^{-1}$ to the right. 
\begin{lemma}\label{lem:Domination} For all $\delta>0$, there exists some constant $k_0=k_0(p,\mu,d,\delta)$ such that for all $k \geq k_0$, 
 there exists a coupling $\mathbf{P}$ and a positive constant $C_{\uparrow\downarrow}$ such that
\begin{align}\label{monotonecouplingB}
   |  \mathbf{E}[\overline{X}_{\tau_k}-\Tilde{Y}_{\tau_k} | \mathcal{E}_{\uparrow\downarrow} ] -C_{\uparrow\downarrow} | \leq \delta + \mathcal{O}(e^{-\lambda}),
\end{align} whereby the implicit constant depends only on $\mu$, $p$, $d$ and $k$, and $\mathbf{E}$ denotes the expectation with respect to $\mathbf{P}$.
\end{lemma}
\begin{proof}
To calculate the speed of $\overline{X}$, let the process evolve together with the TARWDP $\Tilde{Y}$ until time $\T_1$,  denoting the respective coupling by $\mathbf{P}$. After time $\T_1$, we can couple $\overline{X}$ and $\Tilde{Y}$ so that $\overline{X}$ is slower than $\Tilde{Y}$ by letting the environments evolve together and using the same Poisson process for jump attempts with $\overline{X}$, ignoring all points of this process between $\T_1$ and $\T_2$. As in \eqref{eq:TVBound}, for all large enough $k$, we can assume that the state of $e_1$ has the stationary distribution at time $\T_1$, up to an arbitrarily small error $\delta>0$ by choosing $k$ large enough. With positive probability, $\Tilde{Y}$ performs a jump between $\T_1$ and $\T_2$, so $\mathbf{P}[\overline{X}_t\leq \Tilde{Y}_t]=1$ and $\mathbf{P}[\overline{X}_t<\Tilde{Y}_t]>0$. Recall that $(\T_1,\T_2)$ has a 2-dimensional uniform distribution on $[0,\tau_k]$ (with respect to a size-bias for the law of $\tau_k$; see also \eqref{eq:SizeBias}). In particular, the quantity $\mathbf{E}[\T_2-\T_1 | \mathcal{E}_{\uparrow\downarrow}]$ does in the leading order not depend on the bias parameter $\lambda$. Thus, the expected number of \ff-jumps that $\Tilde{Y}$ makes on $[\T_1,\T_2]$ is equal to $$\frac{\mu p}{\mu+1-p}e^{\lambda}Z_{\lambda}^{-1}\mathbf{E}[\T_2-\T_1 | \mathcal{E}_{\uparrow\downarrow}]=\frac{\mu p}{\mu+1-p}\mathbf{E}[\T_2-\T_1 | \mathcal{E}_{\uparrow\downarrow}]+\mathcal{O}(e^{-\lambda}), $$  whereby the implicit constant depends only on $\mu$, $p$, $d$ and $k$. 
Now we have for $\Tilde{Y}$ that the edge $e_1$ is open at time $\T_2$ with probability $\frac{\mu p}{\mu+1-p}$, up to an error of at most $\delta>0$, as $\T_2$ is independent of the dynamics of $\Tilde{Y}$. For $\overline{X}$, the edge $e_1$ is open with  probability $p_{\ast}$ given by 
$$p_{\ast}=p\mathbf{P}[ \T_{\mathbf{u}} \in [\T_1, \T_2] ]+\frac{\mu p}{\mu+1-p}\mathbf{P}[ \T_{\mathbf{u}} \notin [\T_1, \T_2] ].$$ Note that those probabilities depend solely on the update rate of $e_1$ and $\mathbf{E}[\T_2-\T_1]$. Moreover, the coupling $\mathbf{P}$ ensures that if $\Tilde{Y}$ did not make any jumps on $[\T_1,\T_2]$, then the edge $e_1$ to the right of it at time $\T_2$ has the same state as the edge to the right of $X$ at time $\T_2$.
Consequently, 
$$\mathbf{P}[X_t\leq\Tilde{Y}_t]=1\text{ for any }t\in[\T_2,\tau_k] , $$
and there exists a positive constant $C_{\uparrow\downarrow}$ such that 
\begin{align}\label{eq:monotonecouplingB}
    \mathbf{E}[\overline{X}_{\T_2}-\Tilde{Y}_{\T_2}]=-C_{\uparrow\downarrow}+\mathcal{O}(e^{-\lambda}),
\end{align}
whereby the implicit constant depends only on $\mu$, $p$, and $d$. Thus, using the strong Markov property for the RWDPs $X$ and $\Tilde{Y}$ at time $\T_2$, we apply Proposition \ref{speeddisparity} to obtain \eqref{eq:monotonecouplingB} at time $\tau_k$. 
\end{proof}

\subsection{Speed estimate on the critical curve}
Recall the original $\lambda$-RWDP $(\underline{X}_t,\eta_t)_{t\geq 0}$, the reduced $\lambda$-RWDP $(X_t,\eta_t)_{t\geq 0}$, defined in Section~\ref{reducing}, and the one-dimensional TARWDPs $(Y_t,\Tilde{\eta}_t)_{t\geq 0}$ and $(\Tilde{Y}_t,\Tilde{\eta}_t)_{t\geq 0}$, defined in Section~\ref{comparisoncoupling} and Section~\ref{differentcoupling},  respectively. In the following, the implicit constants in $\mathcal{O}$-notation will always  depend only on $\mu$, $p$, $d$ and $k$. We have now all tools to establish the second order asymptotic expansion of the speed of the $d$-dimensional $\lambda$-RWDP on the the critical curve $\mu^2=p(1-p)$.

\begin{proof}[Proof of Theorem \ref{mainlemma2}]
First, we split the probability space into three disjoint events 
\begin{align*}
    \mathcal{E}_{0\bb}&:=\{\text{there is no \bb-points on }[0,\tau_k]\},\\
    \mathcal{E}_{1\bb}&:=\{\text{there is exactly 1 \bb-point on }[0,\tau_k]\},\text{ and}\\
    \mathcal{E}_{\geq2\bb}&:=\{\text{there is at least 2 \bb-points on }[0,\tau_k]\}.
\end{align*}
The expectation $\mathbb{E}[\underline{X}_{\tau_k}]$ can be written as
\begin{align*}
    \mathbb{E}[\underline{X}_{\tau_k}]=\mathbb{E}[\underline{X}_{\tau_k}|\mathcal{E}_{0\bb}]\mathbb{P}[\mathcal{E}_{0\bb}]+\mathbb{E}[\underline{X}_{\tau_k}|\mathcal{E}_{1\bb}]\mathbb{P}[\mathcal{E}_{1\bb}]+\mathbb{E}[\underline{X}_{\tau_k}|\mathcal{E}_{\geq2\bb}]\mathbb{P}[\mathcal{E}_{\geq2\bb}].
\end{align*}
On the event $\mathcal{E}_{0\bb}$, we see that $\underline{X}$ has the same law as $X$ until time $\tau_k$, and hence
\begin{align}\label{E_0b}
\mathbb{E}[\underline{X}_{\tau_k}|\mathcal{E}_{0\bb}]\mathbb{P}[\mathcal{E}_{0\bb}]=\mathbb{E}[X_{\tau_k}]\mathbb{P}[\mathcal{E}_{0\bb}] . 
\end{align} 
On the event $\mathcal{E}_{1\bb}$,  we use Lemma \ref{reducinglemma} to connect the speed of $\underline{X}$ and $X$. Recall the positive constant $C_{1\bb}$ from Lemma \ref{reducinglemma}, and observe that $\underline{X}$ and $X$  satisfy
\begin{align}\label{E_1b}
    \mathbb{E}[\underline{X}_{\tau_k}|\mathcal{E}_{1\bb}]\mathbb{P}[\mathcal{E}_{1\bb}]=(\mathbb{E}[X_{\tau_k}]-C_{1\bb})\mathbb{P}[\mathcal{E}_{1\bb}].
\end{align}
In spirit of \eqref{P(E_1)} and \eqref{P(E_2)}, we compute $\mathbb{P}[\mathcal{E}_{\geq2\bb}]\in\mathcal{O}(e^{-4\lambda})$ and, using the Cauchy-Schwarz inequality, we get 
$\mathbb{E}[\underline{X}_{\tau_k}\mathbbm{1}_{\mathcal{E}_{\geq2\bb}}]
\leq\sqrt{\mathbb{E}[\underline{X}_{\tau_k}^2]}\Pp(\mathcal{E}_{\geq2\bb})$ for all $k\in \mathbb{N}$. 
Combining this with \eqref{E_0b} and \eqref{E_1b}, we obtain 
\begin{align}\label{speedline1b}
    \mathbb{E}[\underline{X}_{\tau_k}]=\mathbb{E}[X_{\tau_k}]-C_{1\bb}\mathbb{P}[\mathcal{E}_{1\bb}]+\mathcal{O}(e^{-3\lambda}).
\end{align}
Splitting the probability space further into events $\mathcal{E}_{\leq 1\oo}$, $\mathcal{E}_{2\oo}\setminus \mathcal{E}_{\uparrow\downarrow}$ and $\mathcal{E}_{\uparrow\downarrow}$, defined in \eqref{eq:o-points}, we get 
\begin{nalign}\label{eq:splitting-o-points}
    \mathbb{E}[\underline{X}_{\tau_k}]=&\mathbb{E}[X_{\tau_k}|\mathcal{E}_{\leq 1\oo}]\mathbb{P}[\mathcal{E}_{\leq 1\oo}]+\mathbb{E}[X_{\tau_k}|\mathcal{E}_{2\oo}\setminus \mathcal{E}_{\uparrow\downarrow}]\mathbb{P}[\mathcal{E}_{2\oo}\setminus \mathcal{E}_{\uparrow\downarrow}]+\mathbb{E}[X_{\tau_k}|\mathcal{E}_{\uparrow\downarrow}]\mathbb{P}[\mathcal{E}_{\uparrow\downarrow}]\\-&C_{1\bb}\mathbb{P}[\mathcal{E}_{1\bb}]+\mathcal{O}(e^{-3\lambda}).
\end{nalign}
In order to evaluate \eqref{eq:splitting-o-points}, we compare $X$ to $Y$ and $\Tilde{Y}$, respectively. 
On the events $\mathcal{E}_{\leq 1\oo}$ and $\mathcal{E}_{2\oo}\setminus \mathcal{E}_{\uparrow\downarrow}$, we compare $X$ to $Y$ using Proposition \ref{monotonecoupling}. Part (i) of Proposition \ref{monotonecoupling} yields that 
\begin{align}\label{eq:1o-point}
\mathbb{E}[X_{\tau_k}|\mathcal{E}_{\leq 1\oo}]=\mathbb{E}[Y_{\tau_k}|\mathcal{E}_{\leq 1\oo}].
\end{align}
By Part (ii) of Proposition \ref{monotonecoupling}, completed by Lemma \ref{lem:Monotone2}(ii), says that on $\mathcal{E}_{2\oo}\setminus \mathcal{E}_{\uparrow\downarrow}$, there exists a sequence $(\delta_k)_{k \in \mathbb{N}}$ such that $\delta_k \rightarrow 0$ as $k\to\infty$ with
\begin{align}\label{2o-points}
   \mathbb{E}[X_{\tau_k}|\mathcal{E}_{2\oo}\setminus \mathcal{E}_{\uparrow\downarrow}]=\mathbb{E}[Y_{\tau_k}|\mathcal{E}_{2\oo}\setminus \mathcal{E}_{\uparrow\downarrow}]-C_{2\oo\setminus \uparrow\downarrow}+\delta_k+\mathcal{O}(e^{-\lambda}). 
\end{align}
On the event $\mathcal{E}_{\uparrow\downarrow}$, we compare $X$ with $\Tilde{Y}$ using Lemma \ref{lem:Domination} to see that 
\begin{align}\label{eq:updown}
    \mathbb{E}[X_{\tau_k}|\mathcal{E}_{\uparrow\downarrow}]=\mathbb{E}[\Tilde{Y}_{\tau_1}|\mathcal{E}_{\uparrow\downarrow}]+C_{\uparrow\downarrow}+\mathcal{O}(e^{-\lambda}).
\end{align}
Plugging \eqref{eq:1o-point}, \eqref{2o-points}, and \eqref{eq:updown} into \eqref{eq:splitting-o-points}, we obtain
\begin{nalign}\label{speedline3}
    \mathbb{E}[\underline{X}_{\tau_k}]=&\mathbb{E}[Y_{\tau_k}|\mathcal{E}_{\leq 1\oo}]\mathbb{P}[\mathcal{E}_{\leq 1\oo}]+\mathbb{E}[Y_{\tau_k}|\mathcal{E}_{2\oo}\setminus \mathcal{E}_{\uparrow\downarrow}]\mathbb{P}[\mathcal{E}_{2\oo}\setminus \mathcal{E}_{\uparrow\downarrow}]+\mathbb{E}[\Tilde{Y}_{\tau_k}|\mathcal{E}_{\uparrow\downarrow}]\mathbb{P}[\mathcal{E}_{\uparrow\downarrow}]\\-&C_{2\oo\setminus \uparrow\downarrow}\mathbb{P}[\mathcal{E}_{2\oo}\setminus \mathcal{E}_{\uparrow\downarrow}]+\delta_k\mathbb{P}[\mathcal{E}_{2\oo}\setminus \mathcal{E}_{\uparrow\downarrow}]-C_{\uparrow\downarrow}\mathbb{P}[\mathcal{E}_{\uparrow\downarrow}]-C_{1\bb}\mathbb{P}[\mathcal{E}_{1\bb}]+\mathcal{O}(e^{-3\lambda}), 
\end{nalign}
whereby $C_{2\oo\setminus \uparrow\downarrow}$, $C_{\uparrow\downarrow}$, and $C_{1\bb}$ are the positive constants introduced in Proposition~\ref{monotonecoupling}, Lemma~\ref{reducinglemma}, and Lemma~\ref{lem:Domination}, respectively.
Notice that in \eqref{speedline3}, the term $\mathbb{E}[\Tilde{Y}_{\tau_k}|\mathcal{E}_{\uparrow\downarrow}]$ appears with coefficient $\mathbb{P}[\mathcal{E}_{\uparrow\downarrow}]$ which can be written as 
\begin{align*}
   \mathbb{P}[\mathcal{E}_{\uparrow\downarrow}]=  c_{\mu,p,d}e^{-2\lambda}+\mathcal{O}_{\mu,p,d}(e^{-3\lambda})
\end{align*}
for some constant $c_{\mu,p,d}>0$. Since $\Tilde{Y}_{\tau_k}$ is independent of $\mathcal{E}_{\uparrow\downarrow}$, we get  $\mathbb{E}[\Tilde{Y}_{\tau_k}|\mathcal{E}_{\uparrow\downarrow}]=\mathbb{E}[\Tilde{Y}_{\tau_k}]$. Next, from Lemma \ref{lemma43} and \cite[Proposition 4.5]{andres2023biased} we know that speed difference of $Y$ and $\Tilde{Y}$ has a leading order of $e^{-2\lambda}$. This observations allow us to rewrite $\mathbb{E}[\Tilde{Y}_{\tau_k}|\mathcal{E}_{\uparrow\downarrow}]\mathbb{P}[\mathcal{E}_{\uparrow\downarrow}]$ as 
\begin{align*}
  \mathbb{E}[\Tilde{Y}_{\tau_k}|\mathcal{E}_{\uparrow\downarrow}]\mathbb{P}[\mathcal{E}_{\uparrow\downarrow}]=   (\mathbb{E}[\Tilde{Y}_{\tau_k}]+\mathcal{C}_{\mu,d,p}e^{-2\lambda}+\mathcal{O}(e^{-3\lambda}))\mathbb{P}[\mathcal{E}_{\uparrow\downarrow}]=\mathbb{E}[Y_{\tau_k}]\mathbb{P}[\mathcal{E}_{\uparrow\downarrow}]+\mathcal{O}(e^{-4\lambda}).
\end{align*}
Recall that $\mathcal{E}_{\leq 2\oo}=\mathcal{E}_{\leq 1\oo}\cup (\mathcal{E}_{2\oo}\setminus \mathcal{E}_{\uparrow\downarrow})\cup \mathcal{E}_{\uparrow\downarrow}$, so 
\begin{align}
    \mathbb{E}[Y_{\tau_k}|\mathcal{E}_{\leq 1\oo}]\mathbb{P}[\mathcal{E}_{\leq 1\oo}]+&\mathbb{E}[Y_{\tau_k}|\mathcal{E}_{2\oo}\setminus \mathcal{E}_{\uparrow\downarrow}]\mathbb{P}[\mathcal{E}_{2\oo}\setminus \mathcal{E}_{\uparrow\downarrow}]+\mathbb{E}[Y_{\tau_k}|\mathcal{E}_{\uparrow\downarrow}]\mathbb{P}[\mathcal{E}_{\uparrow\downarrow}]=\mathbb{E}[Y_{\tau_k}|\mathcal{E}_{\leq 2\oo}]\mathbb{P}[\mathcal{E}_{\leq 2\oo}]
\end{align}
and thus
\begin{nalign}\label{speedline4}
    \mathbb{E}[\underline{X}_{\tau_k}]=&\mathbb{E}[Y_{\tau_k}|\mathcal{E}_{\leq 2\oo}]\mathbb{P}[\mathcal{E}_{\leq 2\oo}]-C_{2\oo\setminus \uparrow\downarrow}\mathbb{P}[\mathcal{E}_{2\oo}\setminus \mathcal{E}_{\uparrow\downarrow}]+\delta_k\mathbb{P}[\mathcal{E}_{2\oo}\setminus\mathcal{E}_{\uparrow\downarrow}]\\&-C_{\uparrow\downarrow}\mathbb{P}[\mathcal{E}_{\uparrow\downarrow}]-C_{1\bb}\mathbb{P}[\mathcal{E}_{1\bb}]+\mathcal{O}(e^{-3\lambda}),
\end{nalign}
where $\delta_k\to 0$ as $k\to\infty$. Since $\mathbb{P}[\mathcal{E}_{\leq 2\oo}]=1-\mathcal{O}(e^{-3\lambda})$, we can approximate $\mathbb{E}[Y_{\tau_k}|\mathcal{E}_{\leq 2\oo}]\mathbb{P}[\mathcal{E}_{\leq 2\oo}]$ by $\mathbb{E}[Y_{\tau_k}]=v_Y\mathbb{E}[\tau_k]=kv_Ye^{1/\mu}$ with an error term in $\mathcal{O}(e^{-3\lambda})$. Here, $v_Y$ denotes the speed of $(Y_t)_{t \geq 0}$.  
Recall that $\mathbb{P}[\mathcal{E}_{2\oo}\setminus \mathcal{E}_{\uparrow\downarrow}]$, $\mathbb{P}[\mathcal{E}_{\uparrow\downarrow}]$, and $\mathbb{P}[\mathcal{E}_{1\bb}]$  have the form $\sum_{i=2}^{\infty}c_ie^{-i\lambda}$ for some sequence $(c_i)_{i \in\N}$. 
Hence, there exists a constant $C_{1\bb\cup 2\oo}>0$ that depends on $\mu$, $p$, $d$ and $k$ such that 
\begin{nalign}\label{speedline9}
    -C_{2\oo\setminus \uparrow\downarrow}\mathbb{P}[\mathcal{E}_{2\oo}\setminus \mathcal{E}_{\uparrow\downarrow}]&-C_{\uparrow\downarrow}\mathbb{P}[\mathcal{E}_{\uparrow\downarrow}]-\frac{\mu p}{\mu+1-p}C_{\mu,p}\mathbb{P}[\mathcal{E}_{1\bb}]+\mathcal{O}(e^{-3\lambda})\\&=-C_{1\bb\cup 2\oo}e^{1/\mu}e^{-2\lambda}+\mathcal{O}(e^{-3\lambda}).
\end{nalign}
This observations allow us to rewrite \eqref{speedline4} as
\begin{align}
    \mathbb{E}[\underline{X}_{\tau_k}]=kv_Ye^{1/\mu}-C_{1\bb\cup 2\oo}e^{1/\mu}e^{-2\lambda}+\delta_k\mathbb{P}[\mathcal{E}_{2\oo}\setminus \mathcal{E}_{\uparrow\downarrow}]+\mathcal{O}(e^{-3\lambda}).
\end{align}
We now turn to compute the speed $v_Y$. Recall that $Y$ is a one-dimensional TARWDP with update rate $\Tilde{\mu}=\mu+p(2d-2)Z_{\lambda}^{-1}$ slowed down by factor of $e^{\lambda}Z_{\lambda}^{-1}$. This means that we can easily compute its speed $v_Y$ using \cite[Lemma 4.4]{andres2023biased} to get that
\begin{align}
    \mathbb{E}[\underline{X}_{\tau_k}]=ke^{\lambda}Z_{\lambda}^{-1}\frac{\Tilde{\mu}e^{-\lambda}Z_{\lambda}p}{\Tilde{\mu}e^{-\lambda}Z_{\lambda}+1-p}e^{1/\mu}-C_{1\bb\cup 2\oo}e^{-2\lambda}+\delta_k\mathbb{P}[\mathcal{E}_{2\oo}\setminus \mathcal{E}_{\uparrow\downarrow}]+\mathcal{O}(e^{-3\lambda}).
\end{align}
Recalling $\mu^2=p(1-p)$ and expanding the first term yields 
\begin{nalign}\label{speedline11}
    \mathbb{E}[\underline{X}_{\tau_k}]&=ke^{1/\mu}\left(\frac{\mu p}{\mu+1-p}-\left(\frac{(2d-2)^2p^2}{\mu+1-p}+\frac{\mu^2p}{(\mu+1-p)^2}+C_{1\bb\cup 2\oo}\right)e^{-2\lambda}\right)\\&+\delta_k\mathbb{P}[\mathcal{E}_{2\oo}\setminus \mathcal{E}_{\uparrow\downarrow}]+\mathcal{O}(e^{-3\lambda}).
\end{nalign}
 From Lemma \ref{prop31}, the speed $v_{\underline{X}}$ of the (original) $\lambda$-RWDP $\underline{X}$ is equal to $\mathbb{E}[X_{\tau_k}]\mathbb{E}[\tau_k]^{-1}$ for any positive integer $k$. Recall Lemma \ref{lemma43} on the asymptotic expansion of $v_{\underline{X}}$. Dividing \eqref{speedline11} by $\mathbb{E}[\tau_k]=ke^{1/\mu}$, and using that $\delta_k\to0$ as $k\to\infty$, and $\mathbb{P}[\mathcal{E}_{2\oo}\setminus \mathcal{E}_{\uparrow\downarrow}] = \mathcal{O}(e^{-2\lambda})$, we derive
\begin{align}\label{speedline10}
    v_{\underline{X}}=\frac{\mu p}{\mu+1-p}-\left(\frac{(2d-2)^2p^2}{\mu+1-p}+\frac{\mu^2p}{(\mu+1-p)^2}+C_{1\bb\cup 2\oo}\right)e^{-2\lambda}+\mathcal{O}(e^{-3\lambda}).
\end{align}
Noticing that $$\frac{(2d-2)^2p^2}{\mu+1-p}+\frac{\mu^2p}{(\mu+1-p)^2}+C_{1\bb\cup 2\oo}>0, $$ this establishes Theorem \ref{mainlemma2}. 
\end{proof}

\bibliographystyle{plain}

\bibliography{biblio}

\begin{thebibliography}{10}

\bibitem{andres2023biased}
Sebastian Andres, Nina Gantert, Dominik Schmid, and Perla Sousi.
\newblock Biased random walk on dynamical percolation.
\newblock {\em The Annals of Probability}, 52(6):2051--2078, 2024.

\bibitem{arous2016biased}
Gerard~Ben Arous and Alexander Fribergh.
\newblock Biased random walks on random graphs.
\newblock {\em Probability and statistical physics in St. Petersburg},
  91:99--153, 2016.

\bibitem{BZ:InvarianceQuenched}
Antar Bandyopadhyay and Ofer Zeitouni.
\newblock Random walk in dynamic {M}arkovian random environment.
\newblock {\em ALEA Lat. Am. J. Probab. Math. Stat.}, 1:205--224, 2006.

\bibitem{barma1983directed}
Mustansir Barma and Deepak Dhar.
\newblock Directed diffusion in a percolation network.
\newblock {\em Journal of Physics C: Solid State Physics}, 16(8):1451, 1983.

\bibitem{berger2003speed}
Noam Berger, Nina Gantert, and Yuval Peres.
\newblock The speed of biased random walk on percolation clusters.
\newblock {\em Probability theory and related fields}, 126(2):221--242, 2003.

\bibitem{bowditch2022biased}
Adam~M Bowditch and David~A Croydon.
\newblock Biased random walk on supercritical percolation: anomalous
  fluctuations in the ballistic regime.
\newblock {\em Electronic Journal of Probability}, 27:1--22, 2022.

\bibitem{chen2019limit}
Dayue Chen, Peng Chen, Nina Gantert, and Dominik Schmid.
\newblock Limit theorems for the tagged particle in exclusion processes on
  regular trees.
\newblock {\em Electronic Communications in Probability}, 24(2):1--10, 2019.

\bibitem{daley2003introduction}
Daryl~John Daley and David Vere-Jones.
\newblock {\em {An introduction to the theory of point processes. Volume I:
  Elementary theory and methods}}.
\newblock Springer, 2003.

\bibitem{fribergh2014phase}
Alexander Fribergh and Alan Hammond.
\newblock Phase transition for the speed of the biased random walk on the
  supercritical percolation cluster.
\newblock {\em Communications on Pure and Applied Mathematics}, 67(2):173--245,
  2014.

\bibitem{gantert2020speed}
Nina Gantert and Dominik Schmid.
\newblock {The speed of the tagged particle in the exclusion process on
  Galton--Watson trees}.
\newblock {\em Electronic Journal of Probability}, 25(71):1--27, 2020.

\bibitem{gu2024random}
Chenlin Gu, Jianping Jiang, Yuval Peres, Zhan Shi, Hao Wu, and Fan Yang.
\newblock {Random walk on dynamical percolation in Euclidean lattices:
  separating critical and supercritical regimes}.
\newblock {\em arXiv preprint arXiv:2407.15162}, 2024.

\bibitem{gu2024speed}
Chenlin Gu, Jianping Jiang, Yuval Peres, Zhan Shi, Hao Wu, and Fan Yang.
\newblock Speed of random walk on dynamical percolation in nonamenable
  transitive graphs.
\newblock {\em arXiv preprint arXiv:2407.15079}, 2024.

\bibitem{olle1997dynamical}
Olle H{\"a}ggstr{\"o}m, Yuval Peres, and Jeffrey Steif.
\newblock Dynamical percolation.
\newblock {\em Annales de l'Institut Henri Poincare (B) Probability and
  Statistics}, 33(4):497--528, 1997.

\bibitem{hermon2020comparison}
Jonathan Hermon and Perla Sousi.
\newblock A comparison principle for random walk on dynamical percolation.
\newblock {\em The Annals of Probability}, 48(6):2952--2987, 2020.

\bibitem{kesten1975limit}
Harry Kesten, Mykyta Kozlov, and Frank Spitzer.
\newblock A limit law for random walk in a random environment.
\newblock {\em Compositio mathematica}, 30(2):145--168, 1975.

\bibitem{levin2017markov}
David Levin, Yuval Peres, and Elisabeth Wilmer.
\newblock {\em Markov chains and mixing times}, volume 107.
\newblock American Mathematical Soc., 2017.

\bibitem{liggett1985interacting}
Thomas Liggett.
\newblock {\em Interacting particle systems}, volume~2.
\newblock Springer, 1985.

\bibitem{liggett1999stochastic}
Thomas Liggett.
\newblock {\em Stochastic interacting systems: contact, voter and exclusion
  processes}, volume 324.
\newblock Springer Science\&Business Media, 1999.

\bibitem{lyons1995ergodic}
Russell Lyons, Robin Pemantle, and Yuval Peres.
\newblock {Ergodic theory on Galton—Watson trees: speed of random walk and
  dimension of harmonic measure}.
\newblock {\em Ergodic Theory and Dynamical Systems}, 15(3):593--619, 1995.

\bibitem{lyons1996biased}
Russell Lyons, Robin Pemantle, and Yuval Peres.
\newblock {Biased random walks on Galton--Watson trees}.
\newblock {\em Probability theory and related fields}, 106:249--264, 1996.

\bibitem{peres2018quenched}
Yuval Peres, Perla Sousi, and Jeffrey Steif.
\newblock Quenched exit times for random walk on dynamical percolation.
\newblock {\em Markov processes and related fields}, 24:715--731, 2018.

\bibitem{peres2020mixing}
Yuval Peres, Perla Sousi, and Jeffrey Steif.
\newblock Mixing time for random walk on supercritical dynamical percolation.
\newblock {\em Probability theory and related fields}, 176(3):809--849, 2020.

\bibitem{peres2015random}
Yuval Peres, Alexandre Stauffer, and Jeffrey Steif.
\newblock Random walks on dynamical percolation: mixing times, mean squared
  displacement and hitting times.
\newblock {\em Probability Theory and Related Fields}, 162(3):487--530, 2015.

\bibitem{sousi2020cutoff}
Perla Sousi and Sam Thomas.
\newblock Cutoff for random walk on dynamical {E}rd{\H{o}}s--{R}{\'e}nyi graph.
\newblock {\em Annales de l’Institut Henri Poincar{\'e}-Probabilit{\'e}s et
  Statistiques}, 56(4):2745--2773, 2020.

\bibitem{sznitman2003anisotropic}
Alain-Sol Sznitman.
\newblock On the anisotropic walk on the supercritical percolation cluster.
\newblock {\em Communications in mathematical physics}, 240:123--148, 2003.

\end{thebibliography}

\end{document}